\documentclass[10pt]{amsart}
\newcommand\ins[1]{#1\xspace}
\newcommand\out[1]{\relax}
\newcommand\fto[2]{\out{#1}\ins{#2}}
\usepackage{amssymb,amsrefs,amsxtra,comment,stmaryrd,xspace}

\newcommand{\de}{{\text{Def}}}
\newcommand{\rde}{{\text{RDef}}}

\newcommand{\K}{{\Kq}\xspace}
\newtheorem{theorem}{Theorem}[section]
\newtheorem{proposition}[theorem]{Proposition}
\newtheorem{corollary}[theorem]{Corollary}
\newtheorem{lemma}[theorem]{Lemma}

\theoremstyle{definition}
\newtheorem{definition}[theorem]{Definition}

\theoremstyle{remark}

\newtheorem{remark}[theorem]{Remark}

\newcommand\sgn{{\operatorname{sgn}}}

\newcommand\ceq{{\, := \,}}
\newcommand\tq{{\, \vert \, }}
\newcommand\ie{\emph{i.e.}}

\DeclareMathOperator\Ad{Ad}

\DeclareMathOperator\ad{ad}

\newcommand\me[1]{{\protect\ensuremath{#1}}\xspace}
\newcommand\mb[1]{\me{\mathbb{#1}}}
\newcommand\mc[1]{\me{\mathcal{#1}}}
\newcommand\mf[1]{\me{\mathfrak{#1}}}

\newcommand\Kq{\mb K}

\newcommand\afort{\emph{a fortiori}\xspace}
\newcommand\loccit{\emph{loc.\ cit.}\xspace}

\newcommand\hecke{{C^\infty_c}}
\newcommand\Hecke{{\mathcal{H}}}
\DeclareMathOperator\vol{vol}

\newcommand\A{\mb A}

\DeclareMathOperator\ord{ord}
\newcommand\ring[1]{\mb{#1}}
\newcommand\C{\ring C} 
\newcommand\Q{\mb Q}
\newcommand\Z{\mb Z} 
\newcommand\R{\ring R}
\newcommand\F{\mb F} 

\newcommand\N{\ring N}

\newcommand\lef{\ring L}
\newcommand\eg{\emph{e.g.},}

\newcommand\rf{\mf k}
\renewcommand\rf{\ensuremath\Bbbk\xspace}

\DeclareMathOperator\U{U}
\DeclareMathOperator\lop{End}

\newcommand\fh{\mathfrak h}
\newcommand\fso{\mathfrak{so}}

\newcommand{\st}{{\text{st}}}
\newcommand\set[2]{\sset{#1\tq#2}}

\newcommand\sset[1]{\{#1\}}
\newcommand\indexme[1]{}
\newcommand\indexmem[1]{}
\DeclareMathOperator\depth{d}

\DeclareMathOperator\supp{supp}

\DeclareMathOperator\Lie{Lie}
\DeclareMathOperator\Int{Int}
\DeclareMathOperator\meas{meas}
\DeclareMathOperator\GL{GL}

\DeclareMathOperator\SL{SL}

\newcommand\BB{\mc B}

\newcommand\fg{\mf g}
\newcommand\ft{\mf t}

\newcommand\wtilde[1]{\widetilde{#1}}

\newcommand\vG{\me{\vec G}}

\newcommand\tR{\me{\wtilde\R}}

\newcommand\inv{^{-1}}

\newcommand\textsub[1]{_{\protect\mathrm{#1}}}

\newcommand\textsup[1]{^{\protect\mathrm{#1}}}
\newcommand\textlsup[1]{\lsup{\protect\mathrm{#1}}}

\newcommand\red{\textsup{red}}

\newcommand\reg{\textsup{reg}}

\newcommand\tame{\textsup{tame}}
\newcommand\unram{\textsup{unr}}

\newcommand\sep{\textsup{sep}}

\newcommand\rBB{\BB\red}
\newcommand\ldef{:=}
\newcommand\lsup[1]{{}^{#1}}

\newcommand\mexp{\me{\mathsf e}}

\newcommand\trans[1]{\textlsup t#1}

\newcommand\dota{\nolinebreak\cdot\nolinebreak} %`action dot'

\newcommand\pref[1]{(\ref{#1})}

\newcounter{tempc}

\newcounter{incenumi} %`Included (inside a theorem) enum i'

\newcounter{tagenumi} %`Tagged enum i'

\makeatletter
\newcommand\nwit[2]{\@addtoreset{#1}{#2}}
\makeatother
\newcommand\term[1]{\textbf{#1}\xspace}
\theoremstyle{remark}

\title[Computability of characters near the identity]
{On the computability of some positive-depth supercuspidal characters near the identity}

\date\today

\author{Raf Cluckers}
\address{Katholieke Universiteit Leuven, Department of Mathematics,
Celestijnenlaan 200B, B-3001 Leu\-ven, Bel\-gium\\ The author is a
postdoctoral fellow of the Fund for Scientific Research - Flanders
(Belgium) (F.W.O.)} \email{raf.cluckers@wis.kuleuven.be}
\urladdr{http://www.wis.kuleuven.be/algebra/Raf/}

\author{Clifton~Cunningham}
\address{Department of Mathematics, University of Calgary}
\email{cunning@math.ucalgary.ca}

\author{Julia~Gordon}
\address{Department of Mathematics, University of British Columbia}
\email{gor@math.ubc.ca}

\author{Loren Spice}
\address{Department of Mathematics, Texas Christian University}
\email{l.spice@tcu.edu}



\begin{document}

\begin{abstract}
This paper is concerned with the values of Harish-Chandra characters 
of a class of positive-depth, toral, very supercuspidal representations of $p$-adic
symplectic and special orthogonal groups, near the identity element.
We declare two representations equivalent if their characters coincide on a 
specific neighbourhood of the identity 
(which is larger than the neighbourhood on which the Harish-Chandra local 
character expansion holds).
We construct a parameter space $B$ (that depends on the group and a real number  $r>0$) 
for the set of equivalence classes 
of the representations of minimal depth $r$ satisfying some 
additional assumptions. 
This parameter space is essentially a geometric object defined over 
$\Q$. Given a non-Archimedean local field $\K$ with sufficiently large residual 
characteristic, the part of the character table near the identity element 
for $G(\K)$
that comes from our class of representations is parameterized 
by the residue-field points of $B$.
The character values themselves can be recovered by specialization from
a constructible motivic exponential function, in the terminology of 
\cite{CLF}. The values of such functions are algorithmically computable.
It is in this sense that we show that a large part of the character table
of the group $G(\K)$ is computable.     

\end{abstract}

\maketitle

\section*{Introduction}
Recent years have seen considerable progress in the program
of studying harmonic analysis on $p$-adic groups in a
field-independent fashion.
This paper takes another step in that program.

On one hand, we are motivated by the hope of an analogy with the case of 
finite groups of Lie type, where for many representations, character values 
``near the identity'' are given by Green polynomials, and where in general the 
characters are recovered from geometric objects. On the other hand,  
 we are guided by the examples in 
\cites{kazhdan-lusztig:fixed-point,hales:characters}
that indicate that for a representation of a $p$-adic group, 
the character values at certain families of
topologically unipotent, regular semisimple elements 
can be expressed in terms of the numbers of 
rational points of a family of hyperelliptic curves
over the residue field.

We consider in this paper
a class of functions that includes both rational functions in
the cardinality of the residue field and 
functions defined by such geometric counting methods; 
namely, the class of constructible exponential functions defined in 
\cite{CLF}.  See \S\ref{sec:our-funs} for details.
The key feature of these functions
is that they can be described in a field-independent manner, by means 
of a formal language of logic. 
Our hope is that Harish-Chandra character of a supercuspidal 
representation  of a $p$-adic group belongs to 
this class of functions.
In this paper, we prove that this is the case for the
restrictions to a neighbourhood of the identity of the characters of a
certain class of toral, very 
supercuspidal, positive-depth representations of symplectic and (split) 
special orthogonal groups. A similar result (stated in slightly different terms) 
was proved in   \cite{gordon:depth0} for a class of
supercuspidal, depth-zero representations (the so called Deligne--Lusztig  representations) of 
symplectic and odd orthogonal groups.
 
Following the program outlined in 
\cite{hales:computed}, we
will prove our result by using the theory of 
motivic integration and the language of definable subassignments.
The specific version of motivic integration that we use here was developed 
by R.~Cluckers and F.~Loeser in \cites{CL, CLF}. We refer to these original
papers for details, and to the expository articles \cites{CL.expo, 
hales:whatis, GY} for an introduction to these ideas and the language. 

When we try to talk about the Harish-Chandra characters of the
representations of a $p$-adic matrix group
in a field-independent way,
we encounter the issue of being unable to identify how we
can keep the representation the ``same'' while varying the
ground field.
More precisely, at present, we do not know if the construction of supercuspidal representations 
	%described in \cite{yu:supercuspidal}
can be carried out within the formal language that we use, without fixing the 
field first; so any question about ``the character of a given 
representation $\pi$'' is not well defined in this context. Instead, we 
construct, for each local, non-Archimedean field \K,
a space, over the residue field of \K,
whose rational points parameterize $r$-equivalence classes of
supercuspidal representations,
in the sense of Definition \ref{def: r-equiv.reps}.
See Definition \ref{def: B}.
The necessity of constructing this parameter space before
even beginning to speak of character values requires us to
impose numerous conditions on the representations that we
consider.

Once we construct the parameter space, we show that
there is a constructible motivic exponential function (whose variables are
the parameters for the representation, and a group element)
whose specialization at \K encodes the restrictions of
the Harish-Chandra characters of the representations of
$G(\K)$ that we consider, up to a rational constant.
See Theorem \ref{thm:character}.

Now we say a few words about the methods of this construction, and the 
assumptions that, at present, we have to impose on the
representations that we consider.
Our main strategy is to express the
restriction of the character to a neighbourhood 
of the identity
as
a Fourier transform of an orbital integral at a semisimple 
element, and then to use the results of \cite{cunningham-hales:good} about good orbital 
integrals.

To make this more precise,
let, for the moment, $\K$ be a fixed local, non-Archimedean field,
and $G$ a connected, reductive, algebraic group defined over
\K.
%We need a suitable ``exponential-type''
%function, so that we can transfer functions on a Lie
%algebra, supported near $0$, to functions on a Lie group,
%supported near $1$; we introduce it in \ref{sec:hyps} (in our cases, 
%Cayley transform satisfies all the requirements for such a map).
%\Lxxx{Julia says: Loren, sorry about using your logo :-)
%I think the following text no longer fits in the introduction *as you have predicted long time ago), I'm looking for a new place for it.} 
In \cite{jkim-murnaghan:gamma-asymptotic}, Kim and Murnaghan
describe the theory of $\Gamma$-asymptotic expansions.  This
is a generalization of Murnaghan--Kirillov theory (see
\cites{murnaghan:chars-classical,adler-debacker:mk-theory})
that
%allows one to describe,
%given \K and $G$ as above and
associates to
a smooth, irreducible representation $\pi$ of $G(\K)$
satisfying certain conditions (see
\cite{jkim-murnaghan:gamma-asymptotic}*{Definition 4.1.3(2)})
%the composition with a suitable ``exponential-type'' map $\mexp$ 
%(see Section \ref{sec:hyps}) of
%the restriction to a (well determined) neighbourhood of the
%identity of the character of $\pi$ as a linear combination
%of Fourier transforms of orbital integrals coming from
%orbits in the Lie algebra of $G(\K)$ that contain in their
%closure a fixed, semisimple element $\Gamma$
%associated to $\pi$
a \K-rational, semisimple element $\Gamma$ in the Lie
algebra of $G(\K)$
(see Theorem 4.4.1 \loccit).
In order to combine these results with those of
Cunningham--Hales \cite{cunningham-hales:good},
we need to impose restrictions on $\Gamma$
(see Definition \ref{defn:rsrg}).
These conditions imply that the representation $\pi$ is
supercuspidal (indeed, toral and ``very supercuspidal''), so
that we may use the explicit construction described by J.-K.~Yu
(see \cite{yu:supercuspidal} and our \S\ref{sub: JKdatum}).
%% Why's it supercuspidal?
%We need to ``complete'' the strongly good,
%positive $G(\K)$-datum to a generic $G(\K)$-datum, in the language of
%\cite{jkim:exhaustion}*{Definitions 5.1, 5.3, and 12.1}.
%This amounts to finding a depth-zero, supercuspidal representation
%of $G^0(\K)$.  In our setting, $G^0(\K)$ is just an elliptic
%torus---so take $1$, or any other depth-zero character!

Once we have constructed the parameter space for the good orbital integrals that arise from the 
characters that we consider, we use motivic integration to show that the {\emph {distribution characters}}, evaluated at families of constructible 
test functions supported by the given neighbourhood of the identity element, 
again produce constructible functions in the parameters.
From this, we proceed to show that the Harish-Chandra character function is, 
in fact, a constructible exponential function, using a
result of J.~Korman on
the local constancy of characters \cite{korman:local-constancy}
(which we reproduce here as Theorem \ref{thm:korman}).
This is carried out in Section \ref{sub: proof thm 3}.
We observe that the local constancy result by itself is not sufficient to conclude that the character is 
a constructible exponential function, and we need all our conclusions about 
the distribution character in addition to it.
In Section \ref{sec:badprimes},
we offer a more general perspective on distributions and
definability, which we expect to be useful in the context of
motivic integration.

To summarize, starting with a symplectic or special orthogonal group $G$, and 
a positive real number $r$, we 
produce a definable subassignment $B_{\fg, -r}$  whose $\rf_{\K}$-points 
parameterize the $r$-equivalence classes of ``restricted'' 
representations (in the sense defined in Section \ref{sub: JKdatum}) of 
$G(\K)$ of minimal depth $r$, where $\rf_\K$ stands for the residue field of $\K$. We then define,
for each $y \in B_{\fg,-r}(\rf_\K)$,
a constructible exponential 
motivic function on the ``subassignment'' of regular topologically 
unipotent elements in $G$ such 
 that the restriction of its specialization to
a neighbourhood of the identity (depending on $r$)
coincides with the restrictions of the Harish-Chandra characters of 
the representations of $G(\K)$ in the $r$-equivalence class
that corresponds to the point $y$.

By definition, 
the values of such constructible exponential functions are in principle algorithmically computable. It is in this sense that we show that in principle, there 
exists an algorithm that generates the part of the character table of $G$ that
corresponds to the character values of restricted representations of a given minimal depth, 
in a neighbourhood of the identity element (which depends on the depth).  

\thanks
The third author is grateful to Fiona Murnaghan for helpful conversations and 
encouragement, and to Max Planck Institute in Bonn for hospitality while a part of this paper was written.

\tableofcontents

\section{Preliminaries}
\label{sec:hyps}

In order to apply the results of
\cite{cunningham-hales:good}, we assume throughout that $G$ is
a symplectic group $G=\operatorname{Sp}_{2N}$
or
a split special orthogonal group   $G=\operatorname{SO}_n$.
In the first case, we write $n_G = 2N$.
In the second, we write $n_G = n$.
Thus, $n_G$ should be thought of as the ``natural matrix
size'' of elements of $G$.
We will denote the Lie algebra of $G$ by \fg,
and the linear dual of \fg by $\fg^*$.

The symbol $\K$ is reserved for a non-Archimedean, local 
field (with no assumption on the characteristic)---\ie,
a finite extension of $\Q_p$, or the field of Laurent
series over a finite field% $\F_q$
%	%% We never use the notation, so (on the basis of
%	%% ``don't name singletons'') I figure that it's safe
%	%% to omit the name.
.
We will henceforth usually drop the adjective
`non-Archimedean'.
We will assume whenever convenient that the residual
characteristic, $p$, of \K is larger than some specified bound.
In particular, we shall always assume that it is odd.

The ring of integers of $\K$ will always be denoted by ${\mathcal O}_\K$,
and the residue field of \K by $\rf_{\K}$.
%(or just $\mathcal O$ and \rf, respectively, if \K is
%understood).
We write $\K\unram$
	(respectively, $\K\tame$)
	%% $\K\sep$ is never used (except in comments).
for an arbitrary unramified
	(respectively, tame)
closure of \K.
If $E/\K$ is an algebraic extension, then there is a unique
valuation, $\ord_\K$, on $E$ such that $\ord_\K(\K) = \Z$.
We call $\ord_\K$ the \term{\K-normalised valuation}
(on $E$).

%Once the field $\K$ is introduced, we will
%assume that $G$ splits over it (which is an additional restriction 
%only when $G=\operatorname{SO}_{2n}$, which we will refer to as ``the even 
%orthogonal case'').
	%% This seems to be misleading, since we work only with
	%% Chevalley groups.
	%% Anyway, we say what we mean more precisely below.

%We also write $G$ for $\bG(\K)$ (and similarly for the
%groups of rational points of other algebraic groups denoted
%by bold, capital, Latin letters).
	%% Sigh ....

In all cases, we realize $G$ explicitly as the group 
$\set{X \in \operatorname{GL}_n}{\trans X J X = J}$,
where $J$ is the matrix satisfying
\begin{equation}
\label{eq:J-symp}
J_{ij}=\begin{cases}
(-1)^i, & i+j=n_G+1 \\
0,      & \text{otherwise,}
\end{cases}
\end{equation}
in the  symplectic case;
and
\begin{equation}
\label{eq:J-orth}
J_{ij}=\begin{cases}
1, & i+j=n_G+1 \\
0, & \text{otherwise,}
\end{cases}
\end{equation}
in the special orthogonal case.
Note that, with this definition, $G$ is automatically \Z-split
in the even special orthogonal case.

Recall that there is associated to $G$
(or any reductive group over \K)
a polysimplicial
$G(\K)$-set $\BB(G, \K)$, the \term{building} of $G$ over \K
(see \cite{bruhat-tits:reductive-groups-1}*{\S7.4}, where
\mc I is used instead),
and, for each $x \in \BB(G, \K)$, a compact, open
subgroup $G(\K)_x$ of $G(\K)$, called a \term{parahoric} subgroup
(see
\cite{bruhat-tits:reductive-groups-2}*{Proposition 4.6.28(i)},
where $P_{\sset x}^0$ is used instead).
Usually one would need to distinguish between the reduced
and enlarged buildings, but, in our situation, $G$ is
semisimple, so there is no difference.

Moy and Prasad have described, for each point
$x \in \BB(G, \K)$, filtrations
\begin{itemize}
\item
of $\fg(\K)$ by compact, open lattices $\fg(\K)_{x, r}$,
\item
of $\fg^*(\K)$ by compact, open subgroups $\fg^*(\K)_{x, r}$, 
and
\item
of $G(\K)_x$ by compact, open, normal subgroups
$G(\K)_{x, r}$
with $G(\K)_x = G(\K)_{x, 0}$
\end{itemize}
(see \cite{moy-prasad:k-types}*{\S\S2.6, 3.2--3.3},
	where $P_{x, r}$ is used instead of $G(\K)_{x, r}$,
and \cite{moy-prasad:jacquet}*{\S3.2--3.3},
	where $\mc G_{x, r}$ is used instead of $G(\K)_{x, r}$).
The indexings of these filtrations depend on the choice of
normalization of valuation.
The index $r$ usually ranges over
$\tR \setminus \sset\infty$
(where the set \tR is as in
\cite{bruhat-tits:reductive-groups-1}*{\S6.4.1}),
but we will find it convenient
to put $\fg(\K)_{x, \infty} = \sset0$,
$\fg^*(\K)_{x, \infty} = \sset0$,
and
$G(\K)_{x, \infty} = \sset1$
for all $x \in \BB(G, \K)$.
These satisfy the obvious $G(\K)$-equivariance
properties, such as that
$\Int(g)G(\K)_{x, r} = G(\K)_{g\dota x, r}$
for all $g \in G(\K)$, $x \in \BB(G, \K)$, and $r \in \tR$.

\begin{remark}
\label{rem:our-B}
In our case, if we realize $G$ as the fixed-point group of
a \K-involution $\sigma$ on $\SL_n$, then there is a natural
involution $\BB(\sigma)$ on
$\BB(\SL_n, \K) = \rBB(\GL_n, \K)$, the reduced building of
$\GL_n$ over \K,
such that $\BB(G, \K)$ may be identified (as a metric space)
with the $\BB(\sigma)$-fixed points in $\rBB(\GL_n, \K)$.
See \cite{kim-moy:involutions}*{Theorem 6.7.3}.
Further, if $x \in \BB(G, \K)$ and $t \in \tR_{> 0}$, then
$G(\K)_{x, t} = G(\K) \cap \GL_n(\K)_{x, t}$.
%	\Lxxx{I'd like to find a reference for this, but
%haven't yet.}
\end{remark}

If $t \in \tR$,
then we write
$\fg(\K)_t = \bigcup_{x \in \BB(G, \K)} \fg(\K)_{x, t}$,
$\fg^*(\K)_t = \bigcup_{x \in \BB(G, \K)} \fg^*(\K)_{x, t}$,
and, if $t \ge 0$,
$G(\K)_t = \bigcup_{x \in \BB(G, \K)} G(\K)_{x, t}$.
If $u \in \tR$ with $t \le u$, then we write
$\fg(\K)_{x, t:u} = \fg(\K)_{x, t}/\fg(\K)_{x, u}$,
$\fg^*(\K)_{x, t:u} = \fg^*(\K)_{x, t}/\fg^*(\K)_{x, u}$,
and, if $t \ge 0$,
$G(\K)_{x, t:u} = G(\K)_{x, t}/G(\K)_{x, u}$.

If $X \in \fg(\K)$, $g \in G(\K)$, and $x \in \BB(G, \K)$,
then we write
\begin{itemize}
\item
$\depth_x(X) = t$ if $t \in \R$ satisfies
$X \in \fg(\K)_{x, t} \setminus \fg(\K)_{x, t^+}$,
and $\depth_x(X) = \infty$ if $X = 0$;
\item
$\depth(X)
= \sup \set{\depth_x(X)}{x \in \BB(G, \K)}$;
\item
$\depth_x(g) = t$ if $t \in \R_{\ge 0}$ satisfies
$g \in G(\K)_{x, t} \setminus G(\K)_{x, t^+}$,
and $\depth_x(g) = \infty$ if $g = 1$;
and
\item
$\depth(g) = \sup \set{\depth_x(g)}{x \in \BB(G, \K), g\in G(\K)_x}$.
\end{itemize}
Notice that we do not define $\depth_x(g)$ if
$g \not\in G(\K)_{x, 0}$, and do not define
$\depth(g)$ if $g \not\in G(\K)_0$.
We call $\depth(X)$ (respectively, $\depth_x(X)$) the
\term{depth} (respectively, \term{$x$-depth}) of $X$, and
similarly for group elements.

Moy and Prasad also define analogues of these objects
when the ground field is changed from \K to an algebraic extension
with finite ramification degree (for example, $\K\unram$);
we will use the obvious notation for these objects.

By definition, an element $g \in G(\K)$ is \term{compact} if the
closed subgroup that it generates is compact.  This is
equivalent to its orbits in $\BB(G, \K)$ being bounded
\cite{bruhat-tits:reductive-groups-1}*{(4.4.9)}, hence, by
\cite{bruhat-tits:reductive-groups-1}*{Proposition 3.2.4},
to $g$ possessing a fixed point in $\BB(G, \K)$.

Let \mexp be the Cayley transform
$X \mapsto (1 + X)(1 - X)\inv$, which identifies the
topologically nilpotent set in $\fg(\K\unram)$ with the
topologically unipotent set in $G(\K\unram)$.
We have that
\begin{itemize}
\item the family of restrictions of \mexp
to the filtration lattices $\fg(\K\unram)_{x, 0^+}$ satisfy
\cite{adler-debacker:mk-theory}*{Property (CE1)}, so that,
for every
$x \in \BB(G, \K\unram)$
and
pair $t, u \in \tR_{> 0}$ with $2 t \ge u$, we have that
\mexp induces an isomorphism
$\mexp_{x, t:u} : \fg(\K\unram)_{x, t:u}
	\to G(\K\unram)_{x, t:u}$;
\item\label{exp-equiv}
the resulting isomorphisms satisfy the
equivariance property
\[
\mexp_{g\dota x, t:u} \circ \Ad(g)
= \Int(g) \circ \mexp_{x, t:u}
\]
for all $x$, $t$, and $u$ as above, and all
$g \in G(\K\unram)$;
and
\item\label{exp-rvalue}
if $T$ is a maximal $\K\unram$-torus in $G$,
$\alpha$ is an absolute root of $T$ in $G$,
$t \in \R_{> 0}$,
$X \in \mf t(\K\unram)_{0^+}$,
and $\ord(d\alpha(X)) = t$,
then
\[
\ord_\K\bigl(
	d\alpha(X)
	- (\alpha(\mexp(X)) - 1)
\bigr) > t.
\]
\end{itemize}
%Note that, for any unramified extension $E/\K$,
%\cite{kim-moy:involutions}*{Theorem 6.7.3} allows us to identify
%$\BB(G, E)$ (as a set) with a subset of
%the reduced building $\rBB(\GL_n, E)$ of $\GL_n$ over $E$.
With the notation of Remark \ref{rem:our-B}, we have that
$\mexp(X) \equiv 1 + 2X \pmod{\GL_n(E)_{x, 2t}}$
whenever $x \in \BB(G, E)$, $t \in \tR_{> 0}$, and
$X \in \fg(E)_{x, t}$. When the residual
characteristic of \K is odd, the first and third facts follow from
easily verified facts about the behaviour of the map
$X \mapsto 1 + X$, since
$\GL_n(E)_{x, 2t} \cap G(E) = G(E)_{x, 2t}$.
The second fact is obvious from the definition.

We will usually be concerned only with the behaviour of
these maps on \K-rational points, but the fact that we can
work with them over $\K\unram$ comes in handy when we
discuss stable conjugacy.

\section{Restricted representations}

In this section, we define the class of representations of $G(\K)$ that
we will consider.

\subsection{Restricted elements}\label{sub:restricted}

We recall a notion introduced in
\cite{adler:thesis}*{Definition 2.2.4}.

\begin{definition}
\label{defn:good}
If $G$ is any connected, reductive group defined over a
local field $\K$, and $\fg$ is the Lie algebra of $G$,
then an element $X \in \fg(\K)$
(respectively, $X^*\in\fg^*(\K)$)
is \term{good of depth $r$} if
\begin{enumerate}
\item\label{defn:good:tame}
there is a maximal $\K$-torus $T$
in $G$, with Lie algebra \ft, that splits over $\K\tame$;
\item\label{defn:good:depth}
$X \in \ft(\K)_r \setminus \ft(\K)_{r^+}$
(respectively, $X^* \in \ft^*(\K)_r \setminus \ft^*(\K)_{r^+}$);
and
\item\label{defn:good:root-value}
for each (absolute) root (respectively, coroot) $\alpha$ of $T$ in $G$,
$d\alpha(X)$ is $0$, or has \K-normalised valuation $r$.
\end{enumerate}
\end{definition}

\begin{remark}
\label{rem:good-depth}
Preserve the notation of Definition \ref{defn:good}.
By \cite{adler-debacker:bt-lie}*{Corollary 3.5.6} (or
its obvious analogue for the dual Lie algebra),
Definition \ref{defn:good}\pref{defn:good:depth} implies that
$r$ is the depth of $X$ in $\fg(\K)$
(respectively, of $X^*$ in $\fg^*(\K)$),
thus justifying the terminology.
For the groups $G$ that we consider, the indices
in the character lattice of $T$
	of the root lattice of $T$ in $G$,
and
in the cocharacter lattice of $T$
	of the coroot lattice of $T$ in $G$,
divide $4$.
Recall that $p \ne 2$.
Thus, by the definition of the filtration on $\ft(\K)$
(see \cite{adler:thesis}*{p.~9}),
if $X$ (respectively, $X^*$) satisfies
Definition \ref{defn:good}\pref{defn:good:root-value}, then
it automatically lies in $\ft(\K)_r$
(respectively, $\ft^*(\K)_r$);
and, in fact,
	it equals $0$
	or
	satisfies Definition \ref{defn:good}\pref{defn:good:depth}.
\end{remark}

The following notion was introduced in
\cite{cunningham-hales:good}*{Definition 2.4},
where the term `slope' was used in place of `depth'.
	%% Oops, it was used here, too, until recently.  :-)

\begin{definition}\label{defn:rsrg}
Let $r$ be a rational number and $\fg$ a classical Lie algebra defined over $\K$.
An element $X$ of $\fg(\K)$
is \term{restricted of depth $r$} in $\fg(\K)$
if it satisfies the following conditions:
\begin{enumerate}
\item\label{defn:rsrg:good}
it is good of depth $r$
(in the sense of Definition \ref{defn:good});
\item\label{defn:rsrg:rss} 
it is regular;
%\item\label{defn:rsrg:tame} 
%$X$ is contained in the Lie algebra of a tamely ramified maximal $\K$-torus $T$ in $H$;
%$X$ is contained in a tamely ramified Cartan subalgebra $\ft(\K)$ defined over $\K$;
%$X$ is a $\K$-rational point on a tamely ramified Cartan $\K$-subalgebra $\ft$ in $\fg$; 
%\item\label{defn:rsrg:root} 
%for each (absolute) root $\alpha$ of $H$ relative to $T$, $\abs{d\alpha(X)}_\K=q^{-r}$;
%for each (absolute) root $\alpha$ of $\fg$ relative to $\ft$, $\abs{\alpha(X)}_\K=q^{-r}$;
	%% These are subsumed by the definition of `good'.
\item\label{defn:rsrg:slope} 
each eigenvalue
is $0$, or has \K-normalised valuation $r$;
and
\item\label{defn:rsrg:multiplicity} 
the multiplicity of the
eigenvalue $0$ is at most $1$.
\end{enumerate}
The set of elements of $\fg(\K)$ that are restricted of
depth $r$ is denoted by $\fg(r,\K)$. 
\end{definition}

\begin{remark}
Note that restricted elements are \afort regular semisimple.
Preserve the notation of Definition \ref{defn:rsrg}.

For \fg a symplectic Lie algebra, if
$\lambda$ is an eigenvalue of $X$, then $2\lambda$ is a root
value for $X$.  Accordingly, in odd characteristic,
Definition \ref{defn:rsrg}\pref{defn:rsrg:rss} implies
Definition \ref{defn:rsrg}\pref{defn:rsrg:multiplicity}; and
Definition \ref{defn:rsrg}\pref{defn:rsrg:slope} implies that,
\emph{if} $X \ne 0$ is good of depth $r$, then it is restricted of 
depth $r$.

For \fg a special orthogonal Lie algebra (even or odd), if
$0$ is an eigenvalue of $X$ with multiplicity at least $2$,
then it is also a root value of $X$.  Accordingly, Definition
\ref{defn:rsrg}\pref{defn:rsrg:rss} implies Definition
\ref{defn:rsrg}\pref{defn:rsrg:multiplicity}.

For \fg an \emph{odd} orthogonal Lie algebra, if $\lambda$ is a
non-\fto{$0$}{zero} eigenvalue of $X$, then it is also a root value of
$X$.  Accordingly, Definition \ref{defn:rsrg}\pref{defn:rsrg:good} implies
Definition \ref{defn:rsrg}\pref{defn:rsrg:slope}.

%For an even orthogonal Lie algebra $\fso_{2n}$ with $p \ge n$,
%and $r$ an integer, the element
%$\diag(0, \varpi^r, \dotsc, (n - 1)\varpi^r,
%	-(n - 1)\varpi^r, \dotsc, -\varpi^r, 0)$
%shows that conditions
%\pref{defn:rsrg:good} and \pref{defn:rsrg:rss}
%can hold without condition \pref{defn:rsrg:multiplicity};
%and the element
%$\diag(\varpi^r, \dotsc, (n - 1)\varpi^r, \varpi^{r + 1},
%	-\varpi^{r + 1}, -(n - 1)\varpi^r, \dotsc, -\varpi^r)$
%shows that they can hold without condition
%\pref{defn:rsrg:slope}.
%On the other hand, if conditions \pref{defn:rsrg:good} and
%\pref{defn:rsrg:rss} hold but condition
%\pref{defn:rsrg:multiplicity} does not, then every
%eigenvalue of $X$ is also a root value, so condition
%\pref{defn:rsrg:slope} holds.

That is, an element of a symplectic or odd special orthogonal Lie
algebra is restricted of depth $r$ if and only if it is
regular, and good of depth $r$.
This implication fails for even special orthogonal Lie algebras.
\end{remark}

%\begin{remark}\label{remark: slopes}
For a given classical Lie algebra $\fg$,
the set of $r\in \Q$ for which $\fg(r,\K)$ is non-empty is independent of $\K$, for sufficiently large residual characteristic.
%It is called the \term{set of slopes} for $\fg$.
%\end{remark}

\subsection{Thickened orbits}
\label{sub:thick}

We recall \cite{cunningham-hales:good}*{Definition 2.5},
which introduces
an equivalence relation on the class of
restricted elements of fixed depth in a Lie algebra.

If $\fg$ is a symplectic or odd special orthogonal Lie algebra,
then two restricted elements $X$ and $X'$ of depth $r$ are
\term{$r$-equivalent} if the multi-sets of eigenvalues
$\set{\lambda_i}{1 \le i \le n_G}$ of $X$
and
$\set{\lambda_i'}{1 \le i \le n_G}$ of $X'$ can be indexed so that
\begin{equation}
\tag{$*$}
\forall 1 \le i \le n_G,\quad   \ord_\K(\lambda_i'-\lambda_i) > r.
\end{equation}
%Note in particular that, by
%Definition \ref{defn:rsrg}\pref{defn:rsrg:slope}, the
%multiplicities of the eigenvalue $0$ of $X$, and of $X'$,
%are the same;
%and that, by Definition \ref{defn:rsrg}\pref{defn:rsrg:multiplicity},
%that multiplicity is $0$ or $1$.
	%% This seemed important to me when I wrote it,
	%% but now I don't know why.
If $\fg$ is
an  even special orthogonal Lie algebra $\fso(2N)$
(so that $n_G = 2N$, in the notation of \S\ref{sec:hyps}),
then we say that $X, X' \in \fg(r,\K)$ are \term{$r$-equivalent} if ($*$)
holds, and, in addition,
\[
   \ord_\K(\operatorname{pfaff}(JX) - \operatorname{pfaff}(J X')) > N r,
\]
where $J$ is as in \eqref{eq:J-orth}
and, for a skew-symmetric matrix $Y$,
$\operatorname{pfaff}(Y)$ is the Pfaffian of $Y$.
 
%\begin{notn}
In all cases, when $X$ is an element of $\fg(r,\K)$, the 
$r$-equivalence class of $X$ is denoted by $[X]_r$.
%\end{notn}

The above definition is somewhat \emph{ad hoc} and
particular to our situation.  We present below an equivalent
definition that will generalize readily to other Lie algebras.

\begin{definition}\label{definition: thick}
Suppose $X\in \fg(\K)$ is regular semisimple and the depth of $X$ is $r$. 
\begin{enumerate}
\item We define
\[
\mathcal{O}_r(X) \ceq \bigcup_{Y\in \ft(\K)_{r^+}}  \mathcal{O}(X+Y),
\]
where $\ft(\K)$ is the Cartan subalgebra containing $X$ and 
where $\mathcal{O}(X+Y)$ is the $G(\K)$-adjoint-orbit of $X+Y$ in 
$\fg(\K)$. We refer to $\mathcal{O}_r(X)$ as the {\bf{thickened orbit}} of $X$.
%Suppose $X\in \fg(\K)$ is regular semisimple and the depth of $X$ is $r$. 
\item We define
\[
\mathcal{O}_r^\st(X) \ceq \bigcup_{Y\in \ft(\K)_{r^+}}  \mathcal{O}^\st(X+Y),
\]
where $\ft(\K)$ is the Cartan subalgebra containing $X$. We refer to $\mathcal{O}_r^\st(X)$ as the {\bf{stable thickened orbit}} of $X$.
\end{enumerate}
\end{definition}

%\Cxxx{The important thing about thickened orbits it not their support but their mesh. Claim: the mesh of $\mathcal{O}_r(X)$ is at least $r^+$; in other words, $1_{\mathcal{O}_r(X)}\in \Hecke_r(\fg(\K))$.
%
%Loren says:  I agree with the spirit of this remark, but
%note that $\mc O_r(X)$ is not compact, so its characteristic
%function does not even lie in $C_c^\infty(\fg(\K))$.}

Lemma~\ref{lemma: thick} assures us that the latter notion is the proper way to understand $r$-equivalent elements in $\fg(r,\K)$.

\begin{lemma}[\cite{cunningham-hales:good}*{Theorem 4.6}]\label{lemma: thick}
For each $X\in \fg(r,\K)$, $\mathcal{O}_r^\st(X) = [X]_r$.
\end{lemma}

\begin{proof}
The proof of this lemma is contained in the proof of 
\cite{cunningham-hales:good}*{Theorem 4.6}.
We explain this in detail in the proof of Proposition \ref{prop: B_g,r}(1)
below,  where we actually prove a more precise statement about 
thickened orbits ({\it vs.}\ stable thickened orbits).
\end{proof}

\subsection{Thickened orbits and orbital integrals}
In \cite{cunningham-hales:good} it is shown how to recognize the stable thickened orbit $\mathcal{O}_r^\st(X)=[X]_r$ as a $\rf_\K$-point on a scheme $S_{\fg,r}$.
As we will see, one of our main tasks in this paper
is to parameterize the thickened orbits of $r$-restricted elements
by the residue-field points of an appropriate object.
To do that, we will use the parameterization of {\emph {stable}} thickened orbits from \cite{cunningham-hales:good}.
In the meantime, the following proposition reveals the relation between thickened orbits of $r$-restricted elements and orbital integrals on $\Hecke_r(\fg(\K))$ (defined below).

\begin{definition}\label{defn: Hecke_r}
Put
\[
\Hecke_r(\fg(\K)) = \sum_{x \in \BB(G,\K)} C_c(\fg(\K)/\fg(\K)_{x, r^+}).
\]
That is, $\Hecke_r(\fg(\K))$
is the space of functions $f \in \hecke(\fg(\K))$ for which there are some $s\leq r$ and a finite set $\{ x_i \tq i\in I\} \subseteq \BB(G,\K)$ such that $f = \sum_{i\in I} f_i$ where $f_i \in C(\fg(\K)_{x_i,s:r^+}$).
	%% I changed C_c^\infty to C, because the _c is implied by the support being in \fg(\K)_{x_i,s}, whereas the ^\infty is implied by invariance under translation by \fg(\K)_{x_i,r^+}.
\end{definition}

\begin{proposition}\label{prop: separates}
Fix $r\in \Q$. Suppose $X, X'\in \fg(r,\K)$. Let $\mu_X : \hecke(\fg(\K)) \to \C$ (respectively, $\mu_{X'} : \hecke(\fg(\K)) \to \C$) be the orbital integral distribution attached to $X$ (respectively, $X'$). Then 
\[
\mathcal{O}_r(X)= \mathcal{O}_r(X') \iff \forall f\in \Hecke_r(\fg(\K)),\ \mu_X(f) = \mu_{X'}(f).
\]
\end{proposition}

\begin{proof}
Suppose $\mu_{X}(f) = \mu_{X'}(f)$ for all $f\in \Hecke_r(\fg(\K))$. 
Let $T$ be the centraliser of $X$ in $G$ and set $\ft = \Lie T$.
	%% I changed $T(\K) = C_{G(\K)}(X)$ to $T = C_G(X)$
	%% because it doesn't affect the group with which we
	%% finish (cf.~Borel, theorem blah), but it does
	%% give us an algebraic group, of which we can take
	%% the Lie algebra.
Suppose $x\in \BB(G,\K)$ lies in the image of $\BB(T,\K)$. (To make sense of this one must observe that $T$ is tamely ramified, since $X\in \fg(r,\K)$.)
Let ${\bf 1}_{X+\fg(\K)_{x,r^+}}$ be the characteristic function of  $X+\fg(\K)_{x,r^+}$. Then ${\bf 1}_{X+\fg(\K)_{x,r^+}} \in \Hecke_r(\fg(\K))$, 
so \[\mu_{X}({\bf 1}_{X+\fg(\K)_{x,r^+}}) = 
\mu_{X'}({\bf 1}_{X+\fg(\K)_{x,r^+}}).\] 
Clearly $\mu_{X}({\bf 1}_{X+\fg(\K)_{x,r^+}})$ is non-zero. 
On the other hand, $\mu_{X'}({\bf 1}_{X+\fg(\K)_{x,r^+}})$ is $0$ unless 
the $G(\K)$-adjoint orbit of $X'$ 
intersects $X + \fg(\K)_{x,r^+}$. 
By \cite{adler:thesis}*{Lemma 2.3.2}, this is true if and
only if the $G(\K)$-adjoint orbit of $X'$ intersects $X + \ft(\K)_{x, r^+}$, in which case $\mc O_{r}(X) = \mc O_{r}(X')$.

Conversely, suppose $\mathcal{O}_r(X)= \mathcal{O}_r(X')$. Because orbital integrals are $G(\K)$-invariant, we may assume $X'\in X+\ft(\K)_{r^+}$, where $\ft(\K)$ is the Cartan subalgebra of $\fg(\K)$ containing $X$, as above.  Let $x$ be as above and observe that $\ft(\K)_{r^+} \subseteq \fg(\K)_{x,r^+}$. Then the image of $X$ under $\fg(\K)_{x,r} \to \fg(\K)_{x,r:r^+}$ coincides with the image of $X'$ under the same function. Arguing as in the proof of \cite{cunningham-hales:good}*{Corollary 1.30}, it follows from \cite{cunningham-hales:good}*{Theorem 1.26}
and the fact that
$|D^{\fg, \mf l}(X)| = |D^{\fg, \mf l}(X')|$
for any Levi \K-subalgebra \mf l of \fg containing
$C_\fg(X) = C_\fg(X')$
(see \cite{cunningham-hales:good}*{(6.1.3)})
that $\mu_X(f) = \mu_{X'}(f)$ for every $f\in \Hecke_r(\fg(\K))$.
\end{proof}

%---------------------------------------------

\subsection{$r$-equivalent representations}\label{sub:r-equiv.reps}

In parallel to the definition of equivalence of elements, we
have a definition of equivalence of representations.

Recall that Moy and Prasad have defined the \term{depth} of
a representation $\pi$ of $G(\K)$ to be the least index $r$ such
that, for some $x \in \BB(G, \K)$, the space of $G\ins{(\K)}_{x, r^+}$-fixed 
vectors in $\pi$ is non-\fto{$0$}{zero}
(see \cite{moy-prasad:k-types}*{Theorem 5.2}).
We write $\depth(\pi)$, rather than $\varrho(\pi)$, for this
depth.

%\begin{definition}\label{def: equiv.reps}
%Let $\pi$ and $\pi'$ be two smooth, irreducible representations of
%$G(\K)$ of depths $r$ and $r'$, respectively.
%Put $R = \max\sset{r, r'}$.
%Then we say that $\pi$ and $\pi'$ are
%\term{$r$-equivalent} if the restrictions of their characters
%to the set of regular, semisimple elements in $G(\K)_R$
%coincide.
%\end{definition}

In this paper, we have occasion to distinguish carefully between
the distribution character of a representation of $G(\K)$ and
its ``density'' 
function on $G(\K)\reg$.
If the field $\K$ has characteristic zero, then it is known
(see Harish-Chandra's theorem \cite{hc:harmonic}) that 
there exists a {\it locally summable} density function $\theta_{\pi}$, 
that is locally constant on $G(\K)\reg$, such that
\begin{equation}\label{eq:HC}
\Theta_{\pi}(f)=\int_{G(\K)}\theta_{\pi}(g)f(g)\,dg
\end{equation}
for {\it all} test functions $f\in C_c^{\infty}(G(\K))$.
However, for fields $\K$ of positive characteristic, it is only 
known that a locally constant function 
$\theta_{\pi}$ on $G(\K)\reg$ exists such that equality (\ref{eq:HC}) 
holds for all test functions
$f$ whose support is contained in $G(\K)\reg$, and it is in general not known 
that this function $\theta_{\pi}$ is locally summable on $G(\K)$. See 
\cite{adler-korman:loc-char-exp}*{Appendix} for a 
discussion of this question, and further references.

We will always
denote 
the distribution character by $\Theta_{\pi}$, and the function
by $\theta_{\pi}$.
By ``Harish-Chandra character'', or just the word
``character'' without modification,
we will always mean the 
function $\theta_{\pi}$. Note, however, that for the groups we are considering, and 
for the fields of positive characteristic, it is only conjectural at this point that 
$\theta_{\pi}$ contains as much information as the distribution character $\Theta_{\pi}$. 
In particular, there is no proof at present 
that if two representations of a general $p$-adic group have the same character
function $\theta_{\pi}$, then their distribution characters coincide.

\begin{definition}\label{def: r-equiv.reps}
Let $\pi$ and $\pi'$ be two smooth, irreducible representations of
$G(\K)$, and assume that $\pi$ has depth $r$. 
%of depth $r$.
We say that $\pi$ and $\pi'$ are
\term{$r$-equivalent} if the restrictions of their characters $\theta_{\pi}$ and 
$\theta_{\pi'}$
to the set of regular, semisimple elements in $G(\K)_r$
coincide.
\end{definition}

While $r$-equivalence in the above sense is obviously weaker
than the usual equivalence of representations (in the sense
of the existence of an invertible intertwining map), it
still preserves some interesting information about a
representation.
For example, two $r$-equivalent representations have the same local
character expansion, hence the same wave-front set
(see \cite{kawanaka:shintani}); in particular, both are, or
both are not, generic
(see \cite{rodier:whittaker-chars}*{%
	Th\'eor\`eme, p.~161
	and
	Remarque 2, pp.~162--163%
}),
and both have the same formal degree.
Since a representation has depth $r$ if and only if there is
some $x \in \BB(G, \K)$ such that the restriction of its
character to $G(\K)_{x, r^+}$ is not orthogonal to $1$ (for
the usual scalar product of functions on a compact group),
$r$-equivalent representations have the same depth.
%\cite{debacker:homogeneity}*{Theorem 3.5.2}
	%% This is the theorem that identifies the range in
	%% which the LCE holds, but we don't actually need
	%% to know that range for this discussion.
There is also known to be interesting arithmetic information
encoded in the behaviour at depth $r$ of the character of a
depth-$r$, supercuspidal representation
of a general linear group; see, for example,
\cite{corwin-moy-sally:gll}*{Theorem 4.2(d)},
\cite{takahashi:gl3}*{Proposition 2.9(2)},
and
\cite{takahashi:gll}*{Theorem 2.5}.
All this information, too, is preserved by $r$-equivalence.
%	\Lxxx{Does equivalence preserve supercuspidality,
%or, more generally, ``supercuspidal support-up-to-$r$-equivalence''?
%That would be neat, but it's probably a question for later.}

\subsection{The Fourier transform}\label{sub: test functions} 

Let $\Lambda$ be an additive character of $\K$ that is
non-trivial on $\mc O_\K$, but trivial on
its unique prime ideal.
We use $\Lambda$ to define a Fourier transform on the space
$\hecke(\fg(\K))$ of compactly supported, locally
constant, complex-valued functions on $\fg(\K)$ by putting
$$
\hat f(X^*)
= \int_{\fg(\K)} f(Y)\Lambda(\langle X^*, Y\rangle)dY
$$
for $f \in \hecke(\fg(\K))$ and $X^* \in \fg^*(\K)$, where
$\langle \cdot, \cdot \rangle$ is the standard pairing between 
$\fg^{\ast}(\K)$ and $\fg(\K)$.
%	\Lxxx{Note that we need more than a character
%to identify the Fourier transform of a function on \fg
%with another such function.  Namely, we need a
%$G(\K)$-invariant pairing.  If we want things to behave well
%on MP filtrations, then the pairing needs to ``respect''
%such filtrations.  Adler and Roche have some results on
%this, but there's no reason to assume more than we need (I
%think!).
%	In the end, we might want to take Fourier transforms
%of functions on $\fg^*(\K)$, rather than on $\fg(\K)$.}

\begin{lemma}[\cite{cunningham-hales:good}*{Lemma 1.8}]\label{lemma: Hecke_r}
With the choices made above, $\Hecke_r(\fg(\K))$ (see Definition~\ref{defn: Hecke_r}) is the space of functions $f \in \hecke(\fg(\K))$ such that the support of $\hat f$ is contained in $\fg^*(\K)_{-r}$.
\end{lemma}

%  \begin{remark}\label{remark: Hecke_r}
%  By
  %turning
%  \cite{cunningham-hales:good}*{Lemma 1.9}% into a definition
%  , the space $\Hecke_r(\fg(\K))$ may be characterised without reference to the Fourier transform.
	%% I think that this re-wording may be clearer to
	%% someone who wasn't privy to our e-mail
	%% discussion.
%  \end{remark}

%the resulting distribution $\res_{\Hecke(r,\K)} \orbital{}{X}{}$ depends only the point $(y_X,\delta_X)$. 
%We will also show (Theorem~\ref{theorem: Wgr}) that 
%for each each $(y,\delta)\in \mathcal{B}^{\ers}_{\fg,r}(\rf_\K)$ 
%there is an elliptic $X\in \fg(r,\K)$ such that 
%$(y,\delta)$ is the pair determined by $\res_{\Hecke(r,\K)}\orbital{}{X}{}$. 
%Before we can prove these results, we must review Waldspurger's parameterization of semi-simple orbits in classical Lie algebras.

As mentioned in \cite{adler-spice:explicit-chars}*{\S1.1},
if $X^* \in \fg^*(\K)$ is semisimple (in the appropriate sense),
then there is a distribution
$\hat\mu_{X^*} : \hecke(\fg(\K)) \to \C$
%(the Fourier transform of a semi-simple orbital integral)
such that, for $f \in \hecke(\fg(\K))$,
$\hat\mu_{X^*}(f)$ is the integral of $\hat f$ along the
$G(\K)$-orbit of $X^*$ with respect to a suitable $G(\K)$-invariant
measure.
By \cite{adler-debacker:mk-theory}*{Theorem A.1.2},
there exists a locally constant function $F$ on
$\fg(\K)\reg$ such that
$\hat\mu_{X^*}(f) = \int_{\fg(\K)} f(X)F(X)dX$
for all $f \in \hecke(\fg(\K)\reg)$.  We will also write
$\hat\mu_{X^*}$ for the function $F$.

\subsection{The representations we consider}\label{sub: JKdatum}
Recall that J.-K.~Yu's construction of supercuspidal
representations of $G(\K)$ (see \cite{yu:supercuspidal})
takes as parameters certain quintuples
$\Psi = (\vG, \vec\phi, \vec r, x, \pi_{-1})$.
Namely,
$\vG = (G^0, \dotsc, G^d = G)$,
$\vec\phi = (\phi_0, \dotsc, \phi_d)$,
and
$\vec r = (0 \le r_0 < \dotsb < r_{d - 1} \le r_d)$,
where, for any $0 \le i < d$,
\begin{enumerate}
\item $G^i$ is
the centraliser of a \K-anisotropic, $\K\tame$-split \K-torus in $G$;
%a \K-subgroup of $G$ such that $G^i \times_\K \K^\sharp$ is a
%$\K^\sharp$-Levi subgroup of G for some tame extension
%$\K^\sharp/\K$
\item $\phi_i$ is a
character of $G^i(\K)$ that is \term{$G^{i + 1}(\K)$-generic} of
depth $r_i$, in the sense of \cite{yu:supercuspidal}*{\S9};
\item $r_0 > 0$ or $d = 0$, and (if $d > 0$)
$r_{d - 1} < r_d = \depth(\phi_d)$ or $\phi_d = 1$;
\item $x$ is a vertex of $\BB(G^0, \K)$;
and
\item $\pi_{-1}$ is a depth-zero, supercuspidal representation of
$G^0(\K)$ that is compactly induced from a representation of the
stabiliser of $G(\K)_x$.
\end{enumerate}
Such a quintuple is called a \term{cuspidal datum} (or
sometimes a generic datum).
The ingredients $(\vec r, x)$ in the above datum are redundant
(see \cite{yu:supercuspidal}*{Theorem 3.1 and Lemma 3.3}),
but we find it convenient nonetheless to have
them available.
We will write $G^{i, \Psi}$ for $G^i$,
$\phi_{i, \Psi}$ for $\phi_i$,
and similar notation for the other ingredients, as
necessary.
Write $\pi_\Psi$ for the representation constructed from $\Psi$
in \cite{yu:supercuspidal}*{\S4}.
Then $\depth(\pi_\Psi) = r_d$.
We will also call $r_d$ the depth of the datum $\Psi$, and
write $\depth(\Psi) = r_d$.
The representations arising via Yu's construction are now
conventionally called \term{tame supercuspidals}.
By \cite{jkim:exhaustion}*{Theorem 19.1}, all
supercuspidal representations of $G(\K)$ are tame, as long
as \K has sufficiently large residual characteristic (see
\S3.4 \loccit, especially Remark 3.5 there).
We will consider only those tame supercuspidals that are
not twists of depth-zero representations (\ie, for which
$d > 0$).
In this case, the number $r_d$ is less interesting than the
number $r_{d - 1}$, which is the smallest depth of a
twist of $\pi_\Psi$ by a character of $G(\K)$.
We call $r_{d - 1}$ the \term{essential depth} of
$\Psi$ (or $\pi_\Psi)$, and
say that $\Psi$ (or $\pi_\Psi$) has \term{minimal depth $r$} if
$r_{d - 1} = r_d = r$.

\begin{remark}
%In fact, since the only essential fact in harmonic analysis that we use is 
%Theorem \ref{thm:main-ext-ingredient}, and since this result is available in the
%depth-zero, ``Deligne--Lusztig'' case (cf.~\cite{debacker-reeder:depth-zero-sc} and
%\cite{kazhdan-varshavsky:depth-zero}), our results also hold in that setting. 
%\Lxxx{Can we make these references more precise?
%The Kazhdan--Varshavsky reference is to their arXiv
%preprint, and should probably be made to point to one of the
%published versions.}
The depth-zero analogue of our Theorem
\ref{thm:distribution characters}(3) was proven in
\cite{gordon:depth0} by a different method.
\end{remark}

\begin{remark}
We will usually identify $\fg^i$ and $\mf z(\fg^i)$ with
subalgebras of \fg, as follows.
If $S^i$ is a
%\K-anisotropic, $\K\tame$-split \K-
torus such that $G^i = C_G(S^i)$, then we regard $\fg^i$ as
the algebra of fixed points for the adjoint action of $S^i$,
and $\mf z(\fg^i)$ as the algebra of fixed points for the
adjoint action of $G^i$.
We regard $\fg^{i\,*}$ and $\mf z(\fg^i)^*$ as subgroups of
$\fg^*$ in a similar fashion.
\end{remark}

\begin{remark}
\label{rem:short-cuspidal-datum}
If $d = 1$ and $G^0$ is an (elliptic) torus, then
$\pi_{-1}$ is a (linear) character.  We may, and do, assume,
upon replacing $\phi_0$ by $\pi_{-1} \otimes \phi_0$,
that it is the trivial character.
We will then call the pair $(G^0, \phi_0)$ a
\term{toral, very supercuspidal datum}, and write just
$\pi = \pi_{(G^0, \phi_0)}$ for the associated
representation.
\end{remark}

For $0 \le i < d$, we can associate to the character $\phi_i$
of $G^i(\K)$ an element
$X^*_i = X^*_{i, \Psi} \in \fg^{i\,*}(\K)_{x, -r_i}$
such that
$$
\phi_i(\mexp_{x, {r_i}:{r_i^+}}(Y))
= \Lambda(\langle X^*_i, Y\rangle)
$$
for all $Y \in \fg^i(\K)_{x, r_i}$.
The condition of genericity for $\phi_i$ implies that
we may choose $X^*$ in the Lie algebra of a $\K\tame$-split,
maximal \K-torus $T \subseteq G^i$ in such a way that the
roots of $T$ in $G^{i + 1}$ that vanish on $X^*_i$
are precisely those that appear in $G^i$ (see
\cite{yu:supercuspidal}*{\S8--9}---in particular, Lemma
8.1 \loccit).
In particular, $X^*_i \in \mf z(\fg^i)^*(\K)$.
Notice that $X^*_i$ is well determined only up to
translation by $\mf z(\fg^i)^*(\K)_{x, -r_i^+}$.

\begin{definition}
\label{defn:X*_Psi}
If $\Psi$ is a cuspidal datum such that
$d = d_\Psi > 0$, then we will write $X^*_\Psi$ for
$\sum_{i = 0}^{d - 1} X^*_{i, \Psi}$, with the
understanding that it is well defined only up to suitable
translation.
We assume that the residual characteristic of \K is sufficiently large,
so that there exists a ``nice'' identification of $\fg(\K)$ and
$\fg^*(\K)$.
(See
\cite{adler-roche:intertwining}*{Proposition 4.1} for
details and a precise statement.)
We fix any such identification,
and write $\Gamma_\Psi$ for the element
of $\fg(\K)$ that corresponds to $X^*_\Psi$.
\end{definition}

The element $\Gamma_\Psi$ in Definition \ref{defn:X*_Psi} is
the same one that occurs in
\cite{jkim-murnaghan:gamma-asymptotic}*{Definition 4.1.3(2)}.
For convenience, we will sometimes write
$\hat\mu_{\Gamma_\Psi}$ instead of $\hat\mu_{X^*_\Psi}$.

\begin{definition}
\label{definition:res_datum}
If $\Psi$ is a cuspidal datum of essential depth $r$,
then we say that $\Psi$ is \term{restricted} if and only if
the coset
$\Gamma_\Psi
	+ \mf z(\fg^{0, \Psi})(\K)_{x_\Psi, (-r)^+}$
(see Definition \ref{defn:X*_Psi})
contains a restricted element.
When this is the case, we will always assume that
$\Gamma_\Psi$ itself is restricted.
\end{definition}

\begin{remark}
\label{rem:res-sc-1-step}
If $\Psi$ is a restricted, cuspidal datum for $G(\K)$,
then $d_\Psi = 1$ and $G^{0, \Psi}$ is an
elliptic torus in $G$.
Thus we may, and do, replace any such cuspidal datum by its
associated toral, very supercuspidal datum (see Remark
\ref{rem:short-cuspidal-datum}), which we will also call a
restricted datum.
\end{remark}

Suppose that
$\Psi$ and $\dot\Psi$
are cuspidal data such that
$\pi_\Psi \cong \pi_{\dot\Psi}$.
By \cite{hakim-murnaghan:distinguished-tame-sc}*{%
	Definitions 4.19(\textbf{F1}) and 5.2,
	and
	Theorem 6.7%
}, we have,
after replacing $\dot\Psi$ by a $G(\K)$-conjugate
(which does not affect whether or not it is restricted)
and
$x_{\dot\Psi}$ by a translate by an element of
the (tensored-up) rational character lattice
$X_*^\K(Z(G^{0, \dot\Psi})) \otimes_\Z \R
= X_*^\K(Z(G)) \otimes_\Z \R$
(which does not affect the Moy--Prasad subgroups appearing below),
that
\begin{gather*}
d \ldef d_\Psi = d_{\dot\Psi}, \\
G^i \ldef G^{i, \Psi} = G^{i, \dot\Psi}\quad\forall 0 \le i < d, \\
\vec r \ldef \vec r_\Psi = \vec r_{\dot\Psi}, \\
x \ldef x_\Psi = x_{\dot\Psi}, \\
\intertext{and}
\prod_{i = 0}^{d - 1} \phi_{i, \Psi}
= \prod_{i = 0}^{d - 1} \phi_{i, \dot\Psi}
\quad\text{on $G^0(\K)_{x, r_{d - 1}}$.}
\end{gather*}
	%% Specifically:
	%% By Definition F1 applied to $i = d$, we have that
	%% $\chi_d = 1$ on $G^d(\K)_{x, r_{d - 1}}$,
	%% hence on $G^{d - 1}(\K)_{x, r_{d - 1}}$;
	%% and, by the same definition applied to $i = d - 1$,
	%% we have that $\phi_{d - 1} = \dot\phi_{d - 1}\chi_d$
	%% on $G^{d - 1}(\K)_{x, r_{d - 2}}$,
	%% hence that $\phi_{d - 1} = \dot\phi_{d - 1}$ on
	%% $G^{d - 1}(\K)_{x, r_{d - 1}}$.
	%%% The above discussion applied when we wanted
	%%% equality just for the single character
	%%% $\phi_{d - 1}$, not the whole product.
	%%% It should be easily updated.
This implies that
$\Gamma_\Psi$ and $\Gamma_{\dot\Psi}$
are congruent modulo $\fg^0(\K)_{x, -r_{d - 1}^+}$,
hence (since both are central) modulo
$\mf z(\fg^0)(\K)_{x, -r_{d - 1}^+}$.
In particular, $\Psi$ is restricted if and only if
$\dot\Psi$ is.
%Once we remember that we had to replace $\dot\Psi$ by a
%$G(\K)$-conjugate to obtain this equality,
%we see that the
%$G(\K)$-conjugacy class of
%$$
%\Bigl(G^{d - 1},
%x,
%\sum_{i = 0}^{d - 1} X^*_{i, \Psi} + \mf z(\fg^{d - 1})^*(\K)_{x, -r_{d - 1}^+}\Bigr)
%$$
%is determined by the equivalence class of $\pi_\Psi$.
%	\Lxxx{Is this another way of saying that we can
%recover the $r$- (or $(-r)$-) equivalence class of
%$\sum_{i = 0}^{d - 1} X^*_{i, \Psi}$ from the equivalence class of
%$\pi_\Psi$?}
Thus, the following definition makes sense.

\begin{definition}
\label{defn:res-sc}
We say that the tame, supercuspidal representation
$\pi_\Psi$ is \term{restricted} if and only if $\Psi$ is
restricted.
\end{definition}

The characters of restricted, tame, supercuspidal
representations%
---indeed, of all those tame,
supercuspidal representations $\pi_\Psi$ for which
$d = d_\Psi > 0$ and
$G^{d - 1, \Psi}/Z(G)$ is \K-anisotropic---%
were computed in
\cite{adler-spice:explicit-chars}.  If, as here, one is only
interested in character values near the identity,
and if the residual
characteristic of \K is sufficiently large
(see \cite{jkim-murnaghan:gamma-asymptotic}*{\S3.2}---in
particular, Remark 3.2.1 \loccit), then we can also
use the results of \cite{jkim-murnaghan:gamma-asymptotic},
which applies to a more general class of representations
than the ones that we consider.
%(but has a more complicated statement).

\begin{theorem}[%
\cite{jkim-murnaghan:gamma-asymptotic}*{Theorem 4.4.1}
and
\cite{adler-spice:explicit-chars}*{Corollary 6.7}%
]
\label{thm:main-ext-ingredient}
Let $\Psi$ be a restricted, depth-$r$, cuspidal datum for
$G(\K)$, and put $\pi = \pi_\Psi$ and
${\Gamma} = {\Gamma_\Psi}$
%$\Gamma = \Gamma_\Psi$
	%% Superstitious brackets to prevent bad spacing.
(in the notation of Definition \ref{defn:X*_Psi}).
Then
$$
\theta_\pi(\mexp(Y))
= \deg(\pi)\hat\mu_\Gamma(Y)
$$
for all regular, semisimple elements $Y$ of $\fg(\K)_r$,
where $\deg(\pi)$ is the formal degree of $\pi$.
\end{theorem}

%	\Lxxx{
%My original comment was false---$\theta_\pi$ is well
%defined (i.e., independent of the choice of Haar measure on
%$G(\K)$.  Neither $\deg(\pi)$ nor $\hat\mu_\Gamma$ is well
%defined individually, but their product is.
%For any given elliptic $T$, we can normalize
%measures such that $\deg(\pi_\Psi) = 1$ whenever
%$C_G(\Gamma_\Psi) = T$.  Should we do so?
%}

\subsection{Constructible motivic, and constructible exponential, 
functions}
\label{sec:our-funs}
Here we will use the theory of motivic integration and Fourier transform 
as developed in \cites{CL,CLF} (see
\cites{CL.expo,GY} for exposition). 
In particular, we use the following notation and terminology introduced in 
\cites{CL, CLF}.
\begin{itemize}
\item  The notation $h[m,n,r]$ stands for 
the functor from the category of fields 
containing $\Q$ to the category of sets that assigns to each field $L$ the set
$L((t))^m\times L^n\times \Z^r$.

\item All the logical formulas we consider are formulas in the 
language ${\mathcal L}_{\mathcal \Z}$ 
(see \cite{CL.expo}*{\S6.7} or \cite{GY}*{\S5}). 
This is the Denef--Pas language with coefficients in 
$\Z\llbracket t\rrbracket$ in the valued field
sort, combined with Presburger language for $\Z$.
Any non-Archimedean local field $\K$ with a choice of a uniformizer is a 
structure for the language ${\mathcal L}_{\Z}$.

\item The category $\de_{\Q}$ is the category of definable subassignments 
of $h[m,n,r]$ for some nonnegative integers $m,n,r$.

\item For a definable subassignment $S\in \de_{\Q}$, we sometimes consider 
the category $\de_S$ of {\emph {definable subassignments over $S$}}.

\item For a positive integer $M$, let ${\mathcal A}_{\Z, M}$ stand for the collection 
of all finite extensions of $\Q_p$ with $p>M$, and let ${\mathcal B}_{\Z,M}$ stand for the 
collection of all fields $\F_q((t))$ with $q=p^r$ for some positive integer $r$  and $p>M$.
We let  ${\mathcal F}_{\Z, M} ={\mathcal A}_{\Z,M}\cup {\mathcal B}_{\Z,M}$. Define
 ${\mathcal F}_{\Z} =\cup_{M>0} {\mathcal F}_{\Z,M}$, and define 
${\mathcal A}_{\Z}$ and ${\mathcal B}_{\Z}$ similarly.

\item We will use the category $\rde_{\Q}$---this is the category of definable subassignments of $h[0,n,0]$ (that is, only residue-field variables are allowed in the formulas, \ie, the elements of this category are definable in the language of rings). For a subassignment $U\in \rde_{\Q}$, 
given a non-Archimedean local field $\K$ with residue field $\rf_{\K}$, we 
will denote the specialization of $U$ to $\K$ by $U(\rf_{\K})$, to emphasize that
it is the set of residue-field points.

\end{itemize} 

We remind the reader that a formula $\varphi$ 
in ${\mathcal L}_{\Z}$ 
with $m$ free variables of the valued-field sort, $n$ free variables of the 
residue-field sort, and $r$ free variables of $\Z$-sort, and no other free 
variables, defines a 
subassignment of the functor $h[m,n,r]$; on the other hand, there exists
$M>0$ (that depends on $\varphi$) such that given a field 
$\K\in {\mathcal F}_{\Z, M}$  with a choice of a uniformizer, 
$\varphi$ 
can be interpreted to give a subset
of $\K^m\times \rf_\K^n\times \Z^r$,
where $\rf_\K$ is the residue field of $\K$.
For details on the specialization of formulas and subassignments, 
see \cite{CL.expo}*{\S6.7}.
For a definable subassignment $S$ of $h[m,n,r]$, 
we denote by $S_{\K}$ its specialization
to $\K$ (so that $S_{\K}$ is a subset of $\K^m\times \rf_\K^n\times \Z^r$).
 
We will often call a set {\emph{definable}} if it is obtained from a 
definable subassignment by specialization. We call a function definable 
if its graph is a definable set.  

Following \cite{CL}*{Section 2.6}, we define a point of a definable 
subassignment 
$S$ to be a pair $y=(y_0, k(y_0))$, where $k(y_0)$ is a field 
containing $\Q$, and 
$y_0\in S(k(y_0))$. If $f:Z\to S$ is a morphism (\ie, $Z$ is an object of 
$\de_S$), one can 
talk about 
the fibre $Z_y$ of $Z$ at $y$, which is a subassignment of $Z$ defined 
using the graph of $f$ (see \cite{CL}*{\S2.6}). Further, by considering 
the graph again, one can show that there exists $M_f>0$ such that 
the specialization of fibres is well defined for $\K\in{\mathcal F}_{\Z, M_f}$.
A morphism of definable subassignments $f:Z \to S$ as above
induces the map of specializations $f_{\K}:Z_{\K}\to S_{\K}$, and the 
fibres of this map are the specializations of the subassignments $Z_y$.
We will abbreviate the notation, and denote by $Z_{y,\K}$ the fibre of 
$f_\K$ at $y\in S_{\K}$. 

Let $S\in \de_{\Q}$ be a definable subassignment. 
In \cite{CL}, R.~Cluckers and F.~Loeser defined the ring 
${\mathcal C}(S)$ of
constructible motivic 
functions on $S$,
which can be made into a $\Q$-algebra by tensoring with 
$\Q$ over $\Z$. Given a non-Archimedean local field \K with a choice 
of uniformizer, we obtain, after applying specialization, the $\Q$-algebra of 
constructible $p$-adic functions ${\mathcal C}_\K$ on $S_{\K}$ (see 
\cite{CL.expo}*{\S6.7}). By construction, these functions are 
$\Q$-valued.
For a constructible motivic function $F\in {\mathcal C}(S)$, we denote 
by $F_{\K}$ its specialization to $S_{\K}$. This function is well 
defined for fields $\K$ of sufficiently large residue characteristic.

Even though we cannot
discuss the ring of  
constructible motivic functions on a subassignment $S$ in detail here, 
we point out that 
it is made up of two parts. One part is the ring of so called Presburger 
functions with values in the ring
\[
A=\Z\bigl[\lef, \lef^{-1}, \bigl((1-\lef^{-i})^{-1}\bigr)_{i\in \N}\bigr].
\]
Here $\lef$ is a formal symbol. When we pass to the specialization of such a function $f$ to a field $\K\in {\mathcal F}_{\Z, M}$, the symbol $\lef$ specializes to the cardinality $q$ of the residue field $\rf_\K$. Thus, 
$f_\K$ is a function on $S_\K$ with values in the ring
$\Z\bigl[1/q, \bigl((1-q^{-i})^{-1}\bigr)_{i\in \N}\bigr]$.
The other part is the Grothendieck ring of the category $\rde_S$. 
When we specialize to a local field $\K$, an element $Z$ of 
$\rde_S$ specializes to 
an integer-valued function that counts the numbers of $\rf_\K$-points on 
fibres $Z_{s, \K}$. 
We refer to \cite{CL.expo}*{Section 3.2} for details.
There is also a notion of positivity for constrcutible motivic functions.
The positive constructible motivic functions always specialize to 
positive-valued functions.

\begin{remark}
Note that to apply motivic integration, one has to pass to 
Functions (introduced in \cite{CL}), which 
are equivalence classes of functions ``modulo support of smaller dimension''. 
We will, however, deal with individual functions, thinking of them as 
representatives of the corresponding Functions. In our case this should 
not cause any 
confusion, since all the test functions we deal with are 
constant on $p$-adic balls of the same dimension as the ambient space.
\end{remark} 

Further, for a subassignment $S$ as above, the ring of motivic 
constructible exponential functions 
${\mathcal C}^{\exp}(S)$ is defined 
in \cite{CLF}. Given a local field \K of 
sufficiently large
residue characteristic, with uniformizer $\varpi$, and an additive 
character $\Lambda$ satisfying the condition
\begin{equation}\label{eq:D_K}
\Lambda(x)=\exp\left(\frac{2\pi i}{p}{\text{Tr}}_{\rf_\K}(\bar x)\right)
\end{equation}
for $x\in {\mathcal O}_\K$, the elements of this ring specialize to what we will call 
($p$-adic) constructible exponential functions.
Here, $p$ is the characteristic of $\rf_\K$, 
$\bar x\in \rf_\K$ is the reduction of $x$ modulo $\varpi$, and 
$\text{Tr}_{\rf_\K}$ is the trace of $\rf_\K$ over its prime subfield
(see \cite{CHL}*{\S10.2} or \cite{GY}*{\S6.3} for 
an exposition). Note that 
here we make a choice of a square root $i$  of $-1$ in $\C$. 
We observe that this assumption  on the character $\Lambda$  is 
a special case of the assumption we made in Section \ref{sub: test functions}. 
%	\Lxxx{Our condition there would have
%allowed us to replace the numerator $2\pi i$ by $2\pi i j$
%for any integer $j$.  Is that OK?  If so, then it's no
%longer important which $i$ we choose, and I withdraw my
%above comment. Julia says: I'm going to ask Raf. Have to leave the boxes 
%until then.}

Given
a field \K as above, with uniformizer $\varpi$ and
an additive character $\Lambda$ as in (\ref{eq:D_K}), we will consider the $\Q$-algebra 
of functions on 
$S_\K$ generated by the specializations
of motivic constructible exponential functions. Let us denote this algebra by 
${\mathcal C}^{\exp}_{\K, \Lambda}(S_\K)$.
By definition, it contains ${\mathcal C}_\K(S_\K)$.
Note that the elements of this algebra no longer have to be $\Q$-valued.

We will often need to talk about constructible (respectively, 
constructible exponential) functions on the set of $\K$-points of an
algebraic group $H$ or its Lie algebra \mf h. We observe  
that any affine algebraic variety $X$
(for example, $X = H$ or $X = \mf h$)
naturally gives a definable 
subassignment of $h[m,0,0]$ for some $m$, and that $X_\K=X(\K)$ 
(where the symbol $X$ on the left stands for the corresponding 
subassignment, and the one on the right 
for the variety), for all non-Archimedean local 
fields $\K$ of sufficiently large residue characteristic.
It is in this sense that we speak below of constructible functions on $H(\K)$ or
$\mf h(\K)$.

Given a field $\K$ (with a uniformizer $\varpi$, and an 
additive character $\Lambda$), we will use the term  
``constructible function'' 
(respectively, 
``constructible exponential function'') on $S_{\K}$ 
for the elements of  
${\mathcal C}_{\K}(S_\K)$ (respectively, of ${\mathcal C}_{\K,\Lambda}^{\exp}(S_\K)$).

\subsubsection{Some definable sets}
\label{subsub:definable}
%Our main goal is to prove that, if $\pi$ is a restricted,
%supercuspidal representation of $G(\K)$, then Harish-Chandra's
%character function
%$\theta_{\pi}$, restricted to a large (definable) 
%neighbourhood of the identity
%element, is a constructible exponential function (on that neighbourhood).
%This is the content of Theorem \ref{thm:character}.
	
%Since constructible, exponential functions, by
%definition, live on $\fg(\K)$, we should probably clarify
%what we mean.  I think that what we mean is that it agrees
%on that large set with a constructible, exponential
%function; but another reasonable (to me) interpretation
%would be that the extension by $0$ of the restriction is a
%constructible, exponential function.
%
%	Oh, and, we probably mean to consider the
%composition with $\mexp$ of this character.  Does that
%change anything?
%Julia says: has this been addressed above?}

To prepare the ground for our main results, we observe that
some of the maps and sets we are using are definable.

First, both our substitute exponential map $\mexp$ and its inverse are
rational functions in the matrix entries of their arguments.
Therefore,
they both take definable sets to definable sets,
and
composition with either takes constructible functions to constructible functions. 
We will rely on this fact everywhere in this paper, without further mention.

Next, we would like to show that 
the sets $G(\K)_{0^+}\reg$ and $\fg_{0^+}\reg$ are definable.
Since $\mexp(\fg(\K)_{0^+}\reg)=G(\K)_{0^+}\reg$, it suffices to show that
$\fg(\K)_{0^+}\reg$ is definable.
%By \cite{adler-debacker:bt-lie}*{\S3.6} and the definition
%of the filtrations on Cartan subalgebras of $\fg(\K)$, if the
%residual characteristic of \K is sufficiently large, then
%a semisimple element of $\fg(\K)$ lies in
%$\fg(\K)_r$ if and only if all of its eigenvalues have
%(\K-normalised) valuation at least $r$.  This is a
%definable condition, since the eigenvalues can be expressed via the values 
%of the roots, and roots can be thought of as the rational functions in the matrix entries.
First, we observe that
%the maximal compact subgroup $G({\mathcal O}_\K)$ is, clearly, definable,
	%% Note that we mustn't call it *the* maximal
	%% compact subgroup, even up to conjugacy;
	%% it is one of many.
%as well as
there is a (field-independent) finite set
$\set{(%
\fg_z,
\overline\fg_z,
\rho_z%
)}{z \in I}$
consisting of triples of:
\begin{itemize}
\item a definable subassignment $\fg_z$ of \fg,
\item a definable subassignment $\overline\fg_z$ of some $h[0,n,0]$, and 
\item a morphism $\rho_z : \fg_z \to \overline\fg_z$
\end{itemize}
such that, for each local, non-Archimedean field \K
and maximal parahoric subgroup $K$ of $G(\K)$,
there exist $z \in I$ and a $G(\K)$-conjugate $K'$ of $K$ such that
the Lie algebra of $K'$ is the specialization of $\fg_z$,
the Lie algebra of the corresponding reductive quotient
is the specialization of $\overline\fg_z$,
and the specialization of $\rho_z$ is the reduction map.
(Here, the indexing set should be thought of essentially as the
set of vertices of some fixed chamber in $\BB(G, \K)$.)
For each $z \in I$,
an element $X$ of the specialization $\fg_{z,\K}$ lies in
$\fg(\K)_{0^+}$ if and only if
$\rho_z(X)^n = 0$ for some integer $n$ that is independent
of \K.
Finally, since the $G(\K)$-orbit of any definable set is definable, we conclude that
$\fg(\K)_{0^+}$, and therefore also $G(\K)_{0^+}$, is definable.  
Since regularity is also definable
(see \cite{gordon-hales:transfer}*{Definition 14}),
we have the desired result.

From now on, we will denote by $G_{0^+}\reg$ (respectively, by $\fg_{0^+}\reg$) 
the definable 
subassignment of $G$ (respectively, of $\fg$) 
that specializes to $G(\K)_{0^+}\reg$ (respectively, $\fg(\K)_{0^+}\reg$).
Note that now we can talk about constructible motivic,
or constructible motivic exponential, functions 
on $G_{0^+}\reg$ and $\fg_{0^+}\reg$.

\subsection{Characters near the identity: the statements}\label{section:S}

Recall that $G$ is a symplectic or split special orthogonal
group.

We need a way to handle distributions (in particular, the distribution characters) in the motivic context. The idea is to define a constructible  
family of constructible motivic test functions, such that 
their specializations form a dense subset of the space of all test functions. 
This  is done in Section \ref{sec:badprimes}. Knowing the values of the 
distribution evaluated at the test functions from this family is 
equivalent to knowing the distribution itself.

%\Lxxx{I still think that we need to be careful below about the
%difference between depth (which can change under central
%twists) and essential depth (which cannot).
%I have made the necessary changes, and, if you agree with
%them, then you can just delete this comment.
%In case it's not clear why, I include some justification
%below.

%While we parameterize representations of \emph{essential}
%depth $r$, the Murnaghan--Kirillov theory that we use to
%describe their values holds only at (actual) depth $r$.
%Thus, we need to say that we're doing one of two things:
%	\begin{itemize}
%	\item parameterizing the germs of
%essential-depth-$r$ characters by $B_{\fg, -r}$ (because we
%really can't control the neighbourhood if we don't control
%the (actual) depth);
%	or
%	\item parameterizing the representations for which
%the depth and essential depth are the same (which we are now
%calling minimal depth).
%	\end{itemize}
%This last one is probably both more appropriate and easier.}

%\Lxxx{We sometimes explicitly specify that we consider
%restricted supercuspidal representations, and sometimes drop
%the word `supercuspidal'; but our definition \emph{includes}
%supercuspidality (and, even if it didn't, it would be a
%natural consequence).  Maybe it's appropriate to mention it
%sometimes anyway, for emphasis or clarity.}

%\Jxxx{Julia says: the next paragraph has changed}
We temporarily have the need to define a notion of 
``equivalence of representations'', that, \textit{a priori}, is even less refined than 
$r$-equivalence.  We will prove in Lemma \ref{lem:equivalence}, 
that this is in fact the same notion as our notion of $r$-equivalence.

%It is necessary because we first prove that the distribution character is, 
%in some sense, definable, and derive from that the character is a 
%constructible motivic exponential  function. 
%Thus, we need to talk about distribution character 
%in the context of motivic integration, before we talk about the character 
%in this context. 

%This happens because, as explained 
%below in section \ref{sec:badprimes}, in positive characteristic 
%we do not {\it {a priori}} know that if two distribution 
% characters coincide on definable functions, then they coincide as
% distributions. We use this cruder notion to state part (1) of the next 
%theorem. 
%For the fields of characteristic zero, it easy to see that the two 

%We prove it in Section \ref{subsub:distributions}. 
%We expect ``that definable $r$-equivalence'' that we define below 
%coincides with ``$r$-equivalence'' 
%in positive characteristic as well, and we discuss this issue in 
%section \ref{sub: local integrability}.

\begin{definition}\label{def:definable equivalence}
Let $\K\in {\mathcal F}_{\Z}$. 
We call two representations $\pi$ and $\pi'$  of $G(\K)$ 
{\bf definably $r$-equivalent} 
(or say that they are in the same definable $r$-equivalence class) if 
there exists a constructible 
family of constructible motivic functions $f_a$, indexed by a 
parameter $a$ from some definable subassignment $S$, such that  
the following conditions are satisfied:
\begin{enumerate}
\item\label{def.eq.1} The support of the specialization $f_{a, \K}$ of $f_a$ to $\K$ is 
contained in $G(\K)_r\reg$, for every $a\in S_\K$.
\item\label{def.eq.2} There exists a family of definable subassignments 
${\mathcal K}_n$, such that:
\begin{itemize}
\item $\cup_{n>0}{\mathcal K}_n=G_{0^+}\reg$, and 
\item For every $n>0$, 
 the $\C$-span of the functions from the family $f_{a, \K}$ with 
supports contained in ${\mathcal K}_{n, \K}$  is  
dense in $C_c^{\infty}(G(\K)_r\reg\cap {\mathcal K}_{n, \K})$ 
(with respect to the $\sup$-norm).
\end{itemize}
\item\label{def.eq.3} The distribution characters of $\pi$ and $\pi'$ coincide on the 
functions 
$f_{a, \K}$, \ie, we have
$$\Theta_{\pi}(f_{a, \K})=\Theta_{\pi'}(f_{a, \K})$$
for all $a\in S_\K$.
\end{enumerate} 
\end{definition}

\begin{theorem}\label{thm:distribution characters}
There exists 
an integer $M>0$ such that the following three statements hold for every 
non-Archimedean local field $\K\in {\mathcal F}_{\Z,M}$
and $r > 0$.
\begin{enumerate}
\item\label{thm:distribution characters:B_g,-r}
There exists a definable, in the language of rings, subassignment 
$B_{\fg, -r}\in \rde_{\Q}$ that parameterizes the definable 
$r$-equivalence classes of 
restricted representations of $G(\K)$ of minimal depth $r$, in the sense
that 
there is a one-to-one (well defined) map from the set of definable 
$r$-equivalence 
classes of such
representations onto
$B_{\fg,-r}(\rf_\K)$.

%Moreover, for the fields of characteristic zero, \ie, for 
%$\K\in {\mathcal A}_{\Z, M}$, definable $r$-equivalence coincides 
%with $r$-equivalence.

\item\label{thm:distribution characters:motivic}
%If $\pi$ is a restricted supercuspidal representation of
%minimal depth $r$, and
%$x\in B_{\fg,-r}(\rf_\K)$ is the corresponding point, then
%we have
%a motivic expression for the distribution character of $\pi$ on $G(\K)_r\ins\reg$:
There exists a positive 
constructible motivic function $Q_{\fg, -r}$ on $B_{\fg, -r}$, 
and
for every constructible motivic test function $f$  such that 
the support of its specialization $f_\K$ is contained 
in $G(\K)_r\reg$, there exists a constructible motivic exponential 
function $H^f$
on the definable subassignment $B_{\fg, -r}$,  and a positive integer $M_f$
(which might depend on $f$), 
such that for every field 
$\K$ with residue characteristic bigger than
$M+M_f$ we have 
\[
\frac{1}{\deg(\pi)}\Theta_{\pi}(f_\K)
= \frac{H^f_\K(x)}{Q_{\fg, -r, \K}(x)},
\]
%for some element  ${\mathcal C}^{\exp}_{B_{\fg,-r}}$
for every
restricted, supercuspidal, minimal-depth-$r$ representation $\pi$
with corresponding point
$x\in B_{\fg, -r}(\rf_\K)$.

\item Let $\{f_a\}_{a\in S}$ be a family of constructible motivic functions on $G(\K)$ satisfying Definition \ref{def:definable equivalence} (\ref{def.eq.1}). 
Then $\Theta_{\pi}(f_{a, \K})$ is a specialization of a 
constructible motivic exponential function of $a$.  
\end{enumerate}
\end{theorem}

\begin{remark}
Note that we consider only representations of minimal depth
above.  Since every representation has a twist by a linear
character of $G(\K)$ that
is of minimal depth, this is not a serious restriction.
Since the necessary twist is trivial on a sufficiently small
neighbourhood of the identity, we could remove the
restriction by considering germs of characters (rather than
their restrictions to a specified neighbourhood of the
identity).
Another approach would be to incorporate the character of
$G$ into the language, so that we could use it directly in
our description of character values.
We believe that our present approach is the simplest.
\end{remark}

Let $G$ be a symplectic or special orthogonal group, as above. 
According to the above theorem, there exists $M>0$ such that for 
every
$\K\in {\mathcal F}_{\Z,M}$, the definable $r$-equivalence classes 
(in the sense of Definition~\ref{def:definable equivalence}) of 
restricted, depth-$r$, supercuspidal 
representations of $G(\K)$ may be identified with points
$x\in B_{\fg, -r}(\rf_\K)$.
Hence, if $\pi$ is a restricted,  supercuspidal
representation of $G(\K)$ of minimal depth $r$ whose image in
$B_{\fg,-r}(\rf_\K)$ is $x$, then, 
as we will prove later in Lemma \ref{lem:equivalence}, we may write unambiguously
$\Theta_x$ for the restriction to $G(\K)_r\reg$ of the
distribution character of $\pi$.
We also denote the restriction 
of $\theta_{\pi}$ to $G(\K)_r\reg$ by $\theta_x$.
Finally, we observe that the germ of the character determines the formal
degree of the representation; hence, all representations of $G(\K)$
that lie in the same definable equivalence class corresponding to a point
$x\in B_{\fg, -r}(\rf_\K)$ have the same formal degree, which we will denote 
by $\deg(x)$.
%let $\theta(x,g)$ be the function defined on  
%$B_{\fg, -r}(\rf_\K)\times G(\K)_r$ by
%\begin{equation}\label{eq:theta(x,g)}
%\theta(x,g):=\theta_x(g).
%\end{equation}

\begin{theorem}\label{thm:character}
Let $G$ and $r>0$ be as above.
Let $M$ be the constant from Theorem \ref{thm:distribution characters}, and let $Q_{\fg, -r}$ be the constructible motivic function on $B_{\fg, -r}$ from the 
same theorem.
Then there exists a
constructible motivic exponential function $F$ on  $B_{\fg, -r}\times G_{0^+}\reg$ 
such that,
for every local field $\K\in {\mathcal F}_{\Z, M}$
and every $x \in B_{\fg,-r}(\rf_\K)$,
\[\frac{1}{\deg(x)}\theta_x(g)
=\frac{F_{\K}(x,g)}{Q_{\fg, -r, \K}(x)}\]
for all $g \in G(\K)_r\reg$
(where $F_\K$ is the specialization of $F$).
In particular, each function $\theta_x$ on $G(\K)_r\reg$ coincides with 
the restriction to $G(\K)_r\reg$ of a 
specialization of a constructible motivic exponential function.  
\end{theorem}

\begin{remark}
%<<<<<<< .mine
%Note that given a representation $\pi$ of $G(\K)$, its distribution  
%character $\Theta_{\pi}$ depends on the choce of Haar measure on $G(\K)$, 
%and so does the formal degree of $\pi$ that appears in 
%Theorem \ref{thm:main-ext-ingredient}.
%Once we have 
%fixed Haar measure on $G(\K)$, formal degree is uniquely determined for each 
%$\pi$. 
%The points of $B_{\fg, -r}(\rf_\K)$ determine the Fourier transforms of the
%orbital integrals, which, in turn, determine our equivalence classes of 
%representations. For each point $x\in B_{\fg, -r}(\rf)$, we show that
%the corresponding character $\theta_x$ is a specialization of a 
%constructible motivic function; however, at present we do not discuss whether formal degree $\deg(x)$
%depends on $x$ in a definable way (though we believe this is so), 
%and this is why we divide the character by formal degree 
%in all our formulas. 
%\end{remark}  
%=======
\label{remark:dep-on-haar}
Note that both the formal degree of a representation,
and the Fourier transform of an elliptic orbital integral 
(here we mean not the distribution, but the function that represents it),
depend on the choice of Haar measure on $G(\K)$
(really, on $G(\K)/Z(G)(\K)$ and $G(\K)/T(\K)$, for $T$ an
elliptic \K-torus, respectively; but remember that we are in
the semisimple case).
On the other hand, the product of the two is well defined
(\ie, independent of Haar measure), so that the equality
asserted in Theorem \ref{thm:main-ext-ingredient} makes
sense.
(The Fourier transform also depends on a choice of Haar
measure on $\fg(\K)$, but there is a natural choice in this
setting, namely, the one for which, with the obvious
notation, $\Hat{\Hat f}(X) = f(-X)$ for all
$X \in \fg(\K)$ and $f \in \hecke(\fg(\K))$.)
Therefore, we may, and do,
fix the Haar measure that is compatible with 
motivic integration, that is, the measure that coincides with 
Serre-Oesterl\'e measure on $G({\mathcal O}_\K)$.

Theorem \ref{thm:character} shows that,
for each $x$, the function
$\theta_x$ is a specialization of a constructible motivic exponential
function; but the dependence on the formal degree at present
prevents us from showing that the function $\theta_x$
depends definably on $x$ (though we believe that this is the
case). Here, we only show that $\frac{1}{\deg(x)}\theta_x$ depends definably on $x$.
\end{remark}

These theorems will be proved in Sections \ref{sub: proof thm 2} and
\ref{sub: proof thm 3}--\ref{sub: proof thm 3 positive char}, respectively.

\subsection{Interpretation in terms of $L$-packets}
The first statement in Theorem \ref{thm:distribution characters}
gives a parameterization of restricted, depth-$r$ representations
up to $r$-equivalence, in the sense of
Definition~\ref{def: r-equiv.reps}.
The strategy for proving it is
to express the restriction of the character of such a
representation as a Fourier transform of an orbital integral
(see Theorem \ref{thm:main-ext-ingredient}),
and then to construct the parameter space
$B_{\fg,-r}$
for the orbital integrals 
that arise
using the ideas from \cite{cunningham-hales: good}.
	%% Hey, that's interesting!  {cunningham-hales:good}
	%% and {cunningham-hales: good} both work to refer
	%% to the same label!

Let $\Psi$ be the cuspidal datum that gives rise to $\pi$,
and $\Gamma_\Psi$ the element defined in Definition
\ref{defn:X*_Psi}.
Then
$$
\Theta_{\pi}(f)=\deg(\pi)\hat{\mu}_{\Gamma_{\Psi}}(f\circ \mexp^{-1})=
\deg(\pi)\mu_{\Gamma_\Psi}(\widehat{f\circ \mexp^{-1}}),
$$
 for all $f\in C_c^{\infty}(G(\K)\reg)$. 
 As we will see from the construction of 
the subassignment $B_{\fg, -r}$ in Section~\ref{sub:param}, the 
elements $\Gamma_\Psi$ are grouped according to stable conjugacy at first, 
before being grouped according to conjugacy. 
Therefore,  
it is natural to ask whether one can detect stable conjugacy of the elements 
$\Gamma_\Psi$ in terms of the cuspidal data that gave rise to them. 
This question is addressed in this section. 

In \cite{reeder:sc-pos-depth}, Reeder constructs some candidates
for $L$-packets of positive-depth, supercuspidal representations on
$G(\K)$ in case $G$ is unramified.
By \S\S3 and 6.9 \loccit, such a packet consists of the representations
$\pi_{(T', \phi')}$, where $(T', \phi')$ ranges over the
equivalence class of a toral, very supercuspidal datum $(T, \phi)$
(see Remark \ref{rem:short-cuspidal-datum})
with $T$ a maximal elliptic \K-torus in $G$ that splits over
$\K\unram$.
%	\Lxxx{I think!  Make this more precise (\S6.9
%carries the desired results, if at all, only in heavily
%coded form).
%On further reading, I can't see anything to do but say ``If
%you read these sections carefully, you'll see that it's
%true.''  Oh well.}
Following \cite{debacker-reeder:depth-zero-sc}*{\S9.4},
we introduce the notion of \term{stable conjugacy} in order to
describe the packet.
The following definition would not change if we allowed
$g$ to range over the $\K\sep$-points of $G$, for a fixed
separable closure $\K\sep/\K$, but it is more convenient for
us to phrase it as we have done.

\begin{definition}
\label{defn:toral-conj}
We say that two toral, very supercuspidal data $(T, \phi)$ and
$(\dot T, \dot\phi)$ for $G(\K)$
(respectively, two regular, semisimple elements
$X, X' \in \fg(\K)$)
are \term{stably conjugate} if and only if there exists some
element $g \in G(\K\unram)$ such that
$\Int(g)T(\K) = \dot T(\K)$ and
$\dot\phi \circ \Int(g) = \phi$
(respectively, $\Ad(g)X = X'$).
\end{definition}

\begin{remark}
\label{rem:toral-conj-depth}
Note that, if $g \in G(\K\unram)$ is as in Definition~\ref{defn:toral-conj} and $r \in \tR_{\ge 0}$,
then
$$
\Int(g)\bigl(T(\K)_r\bigr)
\subseteq
	\Int(g)\bigl(T(\K\unram)_r\bigr) \cap \dot T(\K)
= \dot T(\K\unram)_r \cap T(\K)
= \dot T(\K)_r
$$
(by the definition of the filtrations).
The same argument, with the r\^oles of $T$ and $\dot T$
reversed, shows that we have equality; \ie,  that
$\Int(g)\bigl(T(\K)_r\bigr) = \dot T(\K)_r$.
In particular, the depths of
$\phi$ and $\dot\phi$ are the same.
\end{remark}

It is reasonable to believe that a similar parameterization
to Reeder's will produce $L$-packets for toral, very
supercuspidal representations even if the relevant tori
do not split over an unramified extension
(even, given a suitable definition of
stable conjugacy, if the cuspidal datum is not toral!).

\begin{proposition}
Suppose that
${\Psi} = {(T, \phi)}$
and
${\dot\Psi} = {(\dot T, \dot\phi)}$
	%% Superstitious bracketing, as above.
are stably conjugate (see Definition~\ref{defn:toral-conj}), restricted (hence toral, very supercuspidal)
data of essential depth $r$.
Then,
in the notation of Definition \ref{defn:X*_Psi},
$\mc O^\st_{-r}(\Gamma_\Psi)
	= \mc O^\st_{-r}(\Gamma_{\dot\Psi})$.
\end{proposition}

\begin{proof}
Suppose that the conjugacy of $\Psi$ and $\dot\Psi$ is
afforded by the element $g \in G(\K\unram)$.
It suffices to show that, in the notation of Definition
\ref{defn:X*_Psi},
$\Ad^*(g)X^*_\Psi \equiv X^*_{\dot\Psi}
	\pmod{\dot\ft^*(\K)_{(-r)^+}}$.

Put $\dot X{}^*_\Psi = \Ad^*(g)X^*_\Psi$.
The images in $\BB(G, \K)$ of $\BB(T, \K)$ and
$\BB(\dot T, \K)$ are singletons.  Write $x$ and
$\dot x$, respectively, for their elements.
Since $X^*_\Psi \in \ft^*(\K)$, we have that
$\dot X{}^*_\Psi \in \dot\ft^*(\K)$.
Fix $\dot Y \in \dot\ft(\K)_r$.
Reasoning as in Remark \ref{rem:toral-conj-depth} shows that
$Y \ldef \Ad(g)\inv\dot Y$ lies in $\ft(\K)_r$.
Since
$\mexp_{\dot x, r:r^+}(\dot Y)
= \Int(g)\mexp_{x, r:r^+}(Y)$
by our choice of \mexp (which, remember, admits an
equivariance property with respect to conjugation, not just
by $G(\K)$, but also by $G(\K\unram)$%
\ins{; see page \pageref{exp-equiv}}%
), we have by the
definition of the elements $X^*_\Psi$ and $\dot X{}^*_\Psi$ that
\begin{multline*}
\Lambda(\langle X^*_{\dot\Psi}, \dot Y\rangle)
= \dot\phi(\mexp_{\dot x, r:r^+}(\dot Y))
= \dot\phi(\Int(g)\mexp_{x, r:r^+}(Y)) \\
= \phi(\mexp_{x, r:r^+}(Y))
= \Lambda(\langle X^*_\Psi, Y\rangle)
= \Lambda(\langle\dot X{}^*_\Psi, \dot Y\rangle).
\end{multline*}
Since $\dot Y \in \dot\ft(\K)_r$ was arbitrary, we
obtain the desired congruence.
\end{proof}

\section{Distributions and bad primes}\label{sec:badprimes}
We would like to consider distributions in the context of motivic integration. 
In this paper, we express the value of the distribution character 
$\Theta_{\pi}$ at a test function $f\in C_c^{\infty}(G(\K))$ as a 
motivic integral, when the support of $f$ is contained in a sufficiently small neighbourhood 
of the identity element.   Typically, as we perform motivic integration, we 
need to discard a finite number of primes. Therefore, there arises the problem 
that if for different test functions we need to discard different primes, 
then we might not be able to make any conclusion about the distribution on 
the whole for any fields. In this section  
we show that this is in fact 
not the case.

Note that
Theorem \ref{thm:character}, concerning Harish-Chandra characters,
reduces the need to discuss
distribution characters.
We include this section nonetheless, partly because it 
allows us to establish
Theorem \ref{thm:distribution characters}(1), but also because we hope 
that the simple observations explored here will be 
helpful for many similar situations in other contexts. 

\subsection{A dense family of definable functions}
Let $S_d$ be the definable subassignment of $h[d,0,0]$ defined by 
$K\mapsto K\llbracket t\rrbracket^d$.
Here, as everywhere in this paper, we use the language 
with coefficients in $\Z\llbracket t\rrbracket$ in the valued-field sort.

Let us consider the $\C$-algebra generated by the specializations 
of constructible functions, that is, 
$$
{\mathcal C}_{\K, \C}(\A^d(\K)):={\mathcal C}_{\K}(\A^d(\K))\otimes_{\Q}\C.$$
Similarly, for any definable subassignment $S$, we can consider 
$$
{\mathcal C}_{\K, \C}(S_{\K}):={\mathcal C}_{\K}(S_{\K})\otimes_{\Q}\C.$$
(see \cite{CHL}*{\S2.9} for a discussion of why tensoring with $\C$ is 
compatible with motivic integration and the definition of the class of 
integrable functions).

The proposition below holds for a general base field, but we state it for
 $k=\Q$ for simplicity, since this is all we need.

\begin{proposition}\label{prop: badprimes}
Fix a positive integer $d$. Let $\eta_{a}\in{\mathcal C}(S_d)\otimes \C$ 
be the family of constructible 
motivic functions , where %the parameter
$a=(a_1,\dots a_d, r_1\dots, r_d)$ is an element of 
$\Q\llbracket t\rrbracket^d\times \Z^d$, defined by
$$
\eta_{a}(x_1,\dots, x_d)=
\begin{cases}
1, & \ord(x_i-a_i)\ge r_i \quad \text{for}\quad  i=1,\dots, d\\
0 ,  & \text{otherwise.}
\end{cases}
$$
Suppose $\{f_{\alpha}\}_{\alpha\in A}$ 
is a definable family of definable 
functions on $S_d$, parameterized by some definable subassignment $A$.
Then there exists $M>0$ such that 
for all fields $\K\in {\mathcal F}_{\Z, M}$
there is a well-defined specialization $\eta_{a,\K}$ that is a function 
on ${\mathcal O}_{\K}^d$, and
the motivic integral 
$\displaystyle\int_{S_d}f_{\alpha}\eta_{a}$ specializes to
$\displaystyle\int_{S_{d, \K}}f_{\alpha}\eta_{a,\K}$. 
\end{proposition}
\begin{proof}
This proposition follows directly from the specialization theorems 
(with parameters) 
\cite{CL.expo}*{Theorems 6.9 and 7.3},
since we can simply consider the $a_i$ and $r_i$ 
as extra variables of the functions $\eta_{a}$.
\end{proof}

\begin{corollary}\label{cor: distr}
Let $S$ be a definable subassignment of $h[d,0,0]$ for some $d> 0$.
Assume that $S$ is contained in the subassignment defined by $\ord(x_i)\ge N$,
$i=1,\dots, d$, 
for some $N\in \Z$.
Then exists $M$, such that for any 
$\K\in {\mathcal F}_{\Z, M}$ the algebra  
${\mathcal C}_{\K, \C}(S_{\K})$ is dense in the algebra of 
complex-valued continuous functions on $S_{\K}$ (with $\sup$-norm).
In particular, if two distributions defined on $S_{\K}$ and 
continuous with respect to the $\sup$-norm,  
coincide on ${\mathcal C}_{\K, \C}(S_{\K})$, then they are identical. 
\end{corollary}

\begin{proof}
By multiplying $x_i$ by $t^{-N}$ in the case $N<0$, one can make a definable 
embedding $\phi:S\to S_d$, where $S_d$ is the definable 
subassignment $K\mapsto K\llbracket t\rrbracket^d$ as above.
Then it follows directly from the proposition  that the family 
$\{\xi_{a}=\phi^{\ast}(\eta_{a})\}$ separates points in $S_{\K}$.
The statement now follows from Stone--Weierstrass theorem.
\end{proof}

%The next result follows trivially from the previous corollary, but since 
%it can be useful in many situations beyond the present paper...

Though we do not currently have a general framework for 
``constructible motivic distributions'', 
Corollary \ref{cor: distr} allows us to handle many distributions in the 
context of motivic integration.

\begin{remark}
The theory of motivic integration recently developed by 
E.~Hrushovski and D.~Kazhdan \cites{HK,Yimu,Yimu3} 
includes distributions; 
however, it is not yet known if 
this theory specializes to $p$-adic integration in the same way as the 
theory we are presently using. If this specialization were shown, 
it might yield an alternative approach to the prior discussion.
\end{remark}

\subsection{Distribution characters} Now we are ready to prove that 
definable $r$-equivalence of representations (see Definition 
\ref{def:definable equivalence}) coincides with $r$-equivalence (see 
Definition \ref{def: r-equiv.reps}).

\begin{lemma}\label{lem:equivalence} %Fix $r>0$.
Let $\K\in {\mathcal F}_{\Z}$ be a field. 
%There exists $M>0$ such that for all $\K\in {\mathcal F}_{\Z, M}$, and 
Then two $r$-restricted supercuspidal representations
$\pi$ and $\pi'$ of $G(\K)$ 
are $r$-equivalent if and only if they are definably $r$-equivalent. 
\end{lemma} 

\begin{proof}
Obviously, $r$-equivalence implies definable $r$-equivalence.
Let us prove the converse.
Let $\pi$ and $\pi'$ be definably $r$-equivalent representations.
We need to show that the restrictions of the functions 
$\theta_{\pi}$ and $\theta_{\pi'}$ coincide on $G(\K)_r\reg$.
Suppose this is not so, so that there exists $x_0\in G(\K)_r\reg$ such that
$\theta_{\pi}(x_0)\neq \theta_{\pi'}(x_0)$. Since both
$\theta_{\pi}$, and $\theta_{\pi'}$ are locally constant on $G(\K)_r\reg$, 
there exists a neighbourhood $U$ of $x_0$ such that these functions are 
constant on $U$ with different values. 
Let ${\bf 1}_U$ be the characteristic function of $U$. 

By the definition of definable equivalence, there exists a definable compact
set
${\mathcal K}_{n, \K}$ that contains $U$, and such that 
the span of the set of functions from the family $f_{a, \K}$ of Definition 
\ref{def:definable equivalence} with supports 
contained in ${\mathcal K}_{n, \K}$ is dense in 
$C_c^{\infty}(G(\K)_r\reg\cap {\mathcal K}_{n, \K})$.
Thus, on the one hand, 
for every $\varepsilon$, there exists a function 
$f_{\varepsilon}$ in the span of this family, with 
$\supp(f_{\varepsilon})\subset {\mathcal K}_{n, \K}$, such that 
$\Theta_{\pi}(f_{\varepsilon})=\Theta_{\pi'}(f_{\varepsilon})$, and such that
$$\sup_{{\mathcal K}_{n, \K}} |f_{\varepsilon}-{\bf 1}_{U}|<\varepsilon.$$
On the other hand, as $\varepsilon$ gets small,
$\Theta_{\pi}(f_{\varepsilon})$ has to get arbitrarily 
close to $\theta_{\pi}(x_0)\vol(U)$, whereas
$\Theta_{\pi'}(f_{\varepsilon})$ has to get close to 
 $\theta_{\pi'}(x_0)\vol(U)$, leading 
to a contradiction.
\end{proof}

\subsubsection{A definable family of test functions}\label{subsub:thefamily}
Finally, in order to talk about character distributions in the 
motivic context, 
we construct a definable family of definable functions, essentially, made from 
the family of functions 
$\{\eta_{a}\}_{a\in \Q\llbracket t\rrbracket^d\times \Z^d}$ of Proposition \ref{prop: badprimes}, 
to test the distributions for definable equivalence. Here we have to be 
somewhat careful about the number of restrictions we impose on the residue 
characteristic. As shown in Section \ref{subsub:definable}, there is a 
definable subassignment  $G_{0^+}\reg$ of $G$ such that 
${G_{0^+, \K}\reg}=G(\K)_{0^+}\reg$ for all 
$\K\in {\mathcal F}_{\Z, M}$, with some $M>0$ that depends only on $G$. 
There are finitely many (independent of $\K$) 
conjugacy classes of maximal compact subgroups 
in $G(\K)$, and one can choose a definable compact subgroup 
in each conjugacy class (given $G$, it is possible to write down 
explicit conditions on the matrix entries defining these compact subgroups). 
Let $\Omega$ be a definable subassignment of $G$ such that $\Omega_{\K}$ 
is a union of definable compact subgroups, and every compact 
element in $G(\K)$ is conjugate to an element of $\Omega_{\K}$. 
Let $\Omega_n$ be the subassignment of $\Omega$ defined by 
$\ord\left(D^{\fg}(\mexp^{-1}(g)\right)\le n$, where $D^{\fg}$ is Weyl 
discriminant. 
Then $\{\Omega_n\}$ is a definable family of definable compact subsets of 
$\Omega\cap G\reg$.
Finally, let $\{\xi_{a,n}\}$ be the definable family of definable functions
from Corollary \ref{cor: distr} for the subassignment $\Omega_n$. 
Let $M_{\xi}$ be a constant such that 
for all $n>0$, for all $\K\in {\mathcal F}_{\Z, M_{\xi}}$, the family 
$\{\xi_{a,n}\}$ specializes to 
a dense family of functions with supports contained in ${\Omega_{n,\K}}$. 

We will use the family $\{\xi_{a,n}\}$ in Section \ref{sub: proof thm 2}
in order to establish definable equivalence of characters. We will also need

\begin{lemma}\label{lem:thefunctions}
Suppose $\{f_{\alpha}\}_{\alpha\in A}$ is a constructible family of constructible motivic 
functions. 
Then there exists $M\ge M_{\xi}$ (that depends on the family $\{f_{\alpha}\}$), such that 
$\displaystyle\int_G f_{\alpha}\xi_{a,n}$ specializes to $\displaystyle\int_{G(\K)} f_{\alpha, \K}\xi_{a,n, \K}$
for all $\K\in {\mathcal F}_{\Z, M}$.
\end{lemma} 

\begin{proof} 
This follows from the specialization theorems \cite{CL.expo}*{Theorems 6.9 and 7.3}.
\end{proof}

\section{The parameterization of restricted orbital integrals}

\subsection{Waldspurger's parameterization of semi-simple orbits, 
and $r$-reductions}\label{section: Wdatum}
We only sketch this parameterization, referring to
\cite{waldspurger:nilpotent}*{\S I.7} for
details.
We will use the language of stable conjugacy and stable
orbits throughout; see Definition \ref{defn:toral-conj}.

The orbits of
regular semisimple elements in $\fg(\K)$ are parameterized using
quadruples $(I, (F_i/F_i^{\#}),(a_i), (c_i))$, with:
\begin{itemize}
\item A finite set $I$.
\item For each $i\in I$, a finite extension $F_i^{\#}$ of $\K$, and a
degree-$2$ commutative algebra $F_i$ over $F_i^{\#}$. 
In all cases, $F_i$ is either a quadratic 
extension of $F_i^{\#}$ or is isomorphic to a direct sum of two copies 
of $F_i^{\#}$. The set of indices $i$ such that $F_i$ is a field is denoted 
by $I^{\ast}$. This set is crucial for indexing orbits inside a stable orbit.
\item For each $i\in I$, an element $a_i\in F_i^{\times}$ satisfying the 
condition $\tau_i(a_i)=-a_i$, where $\tau_i$ is the
non-trivial involution
on $F_i$ over $F_i^{\#}$, and such that $a_i$ generates $F_i$ over $\K$. 

\item For each $i\in I$, an element $c_i\in F_i^{\times}$ such that 
$\tau_i(c_i)=-c_i$.
\end{itemize}
There are further conditions each datum must satisfy (see \cite{waldspurger:nilpotent}*{\S I.7, (2), (4)}).
Loosely speaking, the extensions $F_i$ and the elements $a_i$ determine the 
stable conjugacy class of $X$, and the elements $c_i$ are responsible for 
individual $\K$-orbits inside a stable orbit.   

To emphasize the r\^ole of the algebras $F_i$ and the set $I^{\ast}$, 
we change the notation slightly, 
and label the data by
$(I, I^{\ast}, (F_i/F_i^{\#}), (a_i), (c_i))$. Following \cite{waldspurger:nilpotent}, we 
assume that $c_i$ is only given for $i\in I^{\ast}$. 

We recall some of the details of the correspondence
(see the proof of \cite{cunningham-hales:good}*{Theorem 4.6}).
Since we will use these facts in the proof of 
Proposition  \ref{prop: B_g,r}(3) below in an essential way, 
we single them  out as a 
proposition.

\begin{proposition}[\cite{cunningham-hales:good}*{Theorem 4.6, 
Corollary 4.5}]\label{proposition:W&r-equivalence} 
Let $\fg$ be a symplectic or split special orthogonal group and let $\K$ be a fixed local field.
Suppose that $X, X'\in \fg(r,\K)$, and let 
\[
(I, I^{\ast}, F_i, (a_i), (c_i)) \qquad \text{and}\qquad
(I', {I'}^{\ast}, F_i', (a_i'), (c_i'))
\]
be the corresponding data. Then
\begin{enumerate}

\item\label{Wdatum1}
%Let $X$ and $X'$ be regular elements, with corresponding data
%$(I, (F_i/F_i^{\#}),(a_i), (c_i))$ and $(I', (F_i'/F_i'^{\#}),(a'_i), (c'_i))$, 
%respectively. 
The element  $X'$  
is stably conjugate to some element $Y$ that lies 
in the centralizer of an element $X$, if and only if  
there exists a bijection $\psi:I\to I'$ such that $F_i$ is 
isomorphic to $F_{\psi(i)}'$, and the isomorphism 
intertwines $\tau_i$ and $\tau'_{\psi(i)}$, for all $i \in I$. 

\item\label{Wdatum3}
If $[X]_r=[X']_r$, then the centralizers of $X$ and $X'$ 
are stably conjugate; 
there exists a bijection $\psi:I\to I'$ such that 
$F_i\cong F'_{\psi(i)}$ as above, and,  
%	\Lxxx{Note $\tau_i$ and $\tau'_{\psi(i)}$ (see
%\ref{Lxxx:inter}).}
if we identify $F'_{\psi(i)}$ with $F_i$, then 
the elements $a_i$, and $a_i'$ 
have the same valuation and angular component for all $i\in I$. 

\item\label{Wdatum2} 
If $X$ and $X'$ are stably conjugate, then, for all $i \in I$,
the isomorphism $F_i \to F_{\psi(i)}'$ above may be
chosen to carry $a_i$ to $a_i'$.

\end{enumerate}

\end{proposition}

\subsection{The parameterizing subassignment for orbits of restricted elements}\label{sub:param}
Our construction relies heavily on the results of \cite{cunningham-hales:good}. 
In particular, we use the following theorem.
\begin{theorem}[\cite{cunningham-hales:good}*{Theorem 2.6}]
There exists $M>0$ and a variety $S_{\fg, r}$ over $\Z[1/M]$ such that for all
$\K\in {\mathcal F}_{\Z, M}$, there is a natural bijection between 
the $r$-equivalence classes in $\fg(r, \K)$,  
and the points of $S_{\fg, r}(\rf_\K)$. 
\end{theorem}

In the language used in this paper, we will think of $S_{\fg, r}$ as a 
subassignment that is definable in the language of rings, and thus is an element
of $\rde_{\Q}$. 

Proposition \ref{proposition:W&r-equivalence} 
implies that if $X \in \fg(r,\K)$,
the cardinality of the  set $I^{\ast}(X)$ depends 
only on the image of $X$ in $S_{\fg, r}(\rf_\K)$. Therefore, for $x\in S_{\fg, r}(\rf_\K)$, we can talk about the cardinality $|I^{\ast}(x)|$.

\begin{lemma}\label{lem:cardinality} 
Let $k$ be a positive integer. 
The set of points $x$ of $S_{\fg, r}$ such that
$|I^{\ast}(x)|=k$ is definable. We denote the corresponding definable subassignment by $A_{k,\fg, r}$.
\end{lemma}
\begin{proof} 
%As explained in \cite{gordon-hales:transfer},
%	\Lxxx{Can we give a more precise reference?}
Note that the cardinality of $I^{\ast}(X)$ equals 
the number of 
irreducible factors in the characteristic polynomial of $X$ that are even 
polynomials (cf. \cite{cunningham-hales:good}*{\S3.3} and \cite{gordon-hales:transfer}*{\S3.2}).
It is easy to see that there is a (rather cumbersome) formula in the 
language of rings defining the set of elements $X$ such that exactly 
$k$ irreducible factors of the characteristic polynomial of $X$ are even.

We can use the coefficients of the characteristic polynomial of $X$ as free 
variables in this formula (since they are expressions in the matrix entries 
of $X$). Let $P_X$ be the characteristic polynomial of $X$. In the spirit of 
\cite{gordon-hales:transfer}, we think of $P_X$ as a collection of terms that 
are expressions in the entries of $X$.
Then our condition is given by the disjunction
of formulas $\varphi_l$, as $l$ ranges from $0$ to $n_G$:
$$
\begin{aligned}
\varphi_l(P_X)&=\text{`}
\exists f_1,\dots, f_k, g_1, \dots, g_l: P_X=f_1\dots f_kg_1\dots g_l\\
&\wedge(f_1,\dots, f_k \text { are even and nonconstant})\\
& \wedge (g_1,\dots, g_l \text{ are 
not even})\text{'}, 
\quad \text{ for } l>0;\\
\varphi_0(P_X)&=\text{`}
\exists f_1,\dots, f_k: P_X=f_1\dots f_k 
\wedge(f_1,\dots, f_k \text { are even and nonconstant})\text{'}
\end{aligned}
$$

Clearly, the condition that a polynomial is even can be expressed by a 
ring formula in its coefficients. 
Note that within the language of rings, one handles
degree-$m$
polynomials as 
$(m + 1)$-tuples of coefficients, 
and therefore each formula quantified over polynomials $f$, in fact should be interpreted as a disjunction of formulas, one for each possible degree of $f$
(up to degree
$n_G$, in the notation of \S\ref{sec:hyps}), 
each with the corresponding number of variables.
\end{proof}

Now we are ready to construct the subassignment that will parameterize 
the thickened orbits of restricted, depth-$r$ elements.

\begin{definition}\label{def: B}
For each  positive integer $k$, let $T_k$ be the subassignment 
of $h[0,0,1]$ defined by the formula 
$\phi(l)=\text{`$1\le l\le k$'}$ 
(so that any specialization of $T_k$  is 
just a set of $k$ elements), and let $D_k$ be the set of all functions from 
$T_k$ to the set of two elements $\{\pm 1\}$.
Let $B_{\fg, r}$ be the subassignment
\begin{equation}\label{eq: def of B}
B_{\fg, r}:=\bigsqcup_{k=1}^{n/2} 
A_{k, \fg, r}\times D_k,
\end{equation}
where $A_{k, \fg, r}$ is as in Lemma \ref{lem:cardinality}
and $n = n_G$, in the notation of \S\ref{sec:hyps}.
Note that
some terms in the disjoint union in \eqref{eq: def of B} may be empty,
and that
$B_{\fg, r}$ is equipped with a natural projection $\pi_r$ 
to $S_{\fg, r}$.
\end{definition}

\begin{remark}\label{remark: pointing to the appendix}
Alternatively, one can make the space parameterizing the orbital integrals 
from the spaces $S_{\fg, \fh, r}$ constructed in \cite{cunningham-hales:good}, where
$\fh$ ranges over endoscopic Lie algebras for $\fg$. 
%This will be done carefully in a future article.
\end{remark}

%Let us begin by recalling the role played by $B_{\fg,r}$.
%Proposition~\ref{prop: B_g,r}(1, 2) explains
%(and make precise what it means to say) that the definable 
%subassignment $B_{\fg,r}$ parameterizes the distributions 
%$\mu_X : \Hecke_r(\fg(\K)) \to \C$ for $X\in \fg(r,\K)$, in a uniform manner. 
%Proposition~\ref{prop: B_g,r}(3) shows that $B_{\fg,r}\times \Hecke_r$ 
%parameterizes definable subassignments $U^f_{[X]_r}$ that determine the 
%\emph{value} of $\mu_X(f)$, for every $X\in \fg(r,\K)$ and for 
%every $f\in \Hecke_r(\fg(\K))$.

%Now suppose $X, X'\in \fg(r,\K)$. 
%It follows from Proposition~\ref{prop: B_g,r} that if the thickened orbit 
%of $X$ is equal to the thickened orbit of $X'$, then $\mu_X$ and $\mu_{X'}$ 
%coincide as distributions on $\Hecke_r(\fg(\K))$. 
%Proposition~\ref{prop: separates} tells us that converse is true:  
%if $X$ and $X'$ are element of $\fg(r,\K)$ and if $\mu_X$ and $\mu_{X'}$ 
%coincide as distributions on $\Hecke_r(\fg(\K))$, then $\mathcal{O}_r(X)= \mathcal{O}_r(X')$.
%Thus, $B_{\fg,r}$ parameterizes $\{ \mathcal{O}_r(X) \tq X \in \fg(r,\K)\}$.

\begin{proposition}\label{prop: B_g,r} 
There exists a morphism of definable subassignments $\nu_r:\fg(r)\to B_{\fg, r}$, and 
exists $M>0$ 
such that for every field $\K\in {\mathcal F}_{\Z, M}$, 
$\nu_r$ specializes to
a surjective
map $\nu_{r,\K}:\fg(r, \K) \to B_{\fg,r}(\rf_\K)$ such that: 
\begin{enumerate}
%\item $\mu(X)=\mu(X')$ if and only if $[X]_r=[X']_r$, and there exists an 
%element $X''\in [X']_r$ such that $X''$ is conjugate to $X$ (over $\K$). 
\item If $[X]_r=[X']_r$, then 
$\pi_{r}\circ \nu_{r,\K}(X)=\pi_r\circ \nu_{r,\K}(X')$.
In fact, $\nu_{r,\K}(X)=\nu_{r,\K}(X')$ if and only 
if $X'\in {\mathcal O}_r(X)$. That is, 
for every $y\in B_{\fg, r}(\rf_\K)$, the set $\nu_{r,\K}^{-1}(y)$ is a 
thickened orbit (see Definition \ref{definition: thick}).
%element $X''\in [X']_r$ such that $X''$ is conjugate to $X$ (over $\K$)

\item
There exists a definable subassignment ${\mathcal O}^r$ in the category
$\de_{B_{\fg, r}}$ that for all fields $\K\in {\mathcal F}_{\Z, M}$
and all $y \in B_{\fg, r}(\rf_\K)$
specializes to the thickened orbits, 
in the sense that
\begin{equation*}\label{eq:definable thick}
{\mathcal O}^r_{y,\K}=\nu_{r, \K}^{-1}(y)=\mathcal O_r(X_y)
\end{equation*}
whenever $X_y\in \fg(r, \K)$ satisfies
$\nu_{r,\K}(X_y)=y$. 

\item The
restriction to $\Hecke_r(\fg(\K))$ of the
orbital integral of $X \in \fg(r,\K)$
%(as a distribution)
	%% This is unnecessary, since, unlike its Fourier
	%% transform, the orbital integral will (usually)
	%% not be representable by a function.
is completely determined by $\nu_{r,\K}(X)$.
That is,
\[
\nu_{r,\K}(X) = \nu_{r,\K}(X') \iff \forall f\in \Hecke_r(\fg(\K)),\ins\ \mu_X(f) = \mu_{X'}(f).
\]
We denote the corresponding
(restricted)
orbital integral by
$\mu_y$, where $y = \nu_{r,\K}(X)$
(so this notation is well defined only for fields $\K$ of sufficiently large residue characteristic).

\item\label{item:Q} 
There exists a positive constructible motivic function $Q_{\fg, r}(y)$ 
on $B_{\fg, r}$, and for every constructible motivic  test function $f$  
such that the Fourier transform of $f_\K$ is supported by $\fg(\K)_{-r}$,
there exists a constructible motivic 
function $\Phi^f$
on the definable subassignment $B_{\fg, r}$, such that for every field 
$\K\in {\mathcal F}_{M+M_f}$ (where $M_f$ might depend on $f$) we have 
\[
\mu_y(f_\K)
= \frac{\Phi^f_\K(y)}{Q_{\fg, r, \K}(y)},
\]
%for some element  ${\mathcal C}^{\exp}_{B_{\fg,-r}}$
for every $y\in B_{\fg, r}(\rf_\K)$, where $\mu_y(f_\K)$ stands for the value of the orbital integral of $f_\K$ along the orbit of $X_y$ for any $X_y$ such that
$\nu_{r,\K}(X_y)=y$.

\end{enumerate}
\end{proposition}

\begin{proof}
Let $\K$ be a local field, and let us take an element  $X\in \fg(r, \K)$.
Let  $(I, I^{\ast}, (a_i), (c_i))$ be the datum attached to $X$.
Recall that $[X]_r = \mathcal{O}_r^\st(X)$ (see Lemma~\ref{lemma: thick}) is a union of stable orbits.
Let 
$x\in S_{\fg, r}(\rf_\K)$ be the point that corresponds to the class $[X]_r$, 
according to \cite{cunningham-hales:good}*{Theorem 4.4}.

We need to choose a Kostant section (a map
from the set of stable adjoint orbits in $\fg\reg(\K)$
to the set of (rational) adjoint orbits in $\fg\reg(\K)$).
This map is associated with a choice of a nilpotent element $N\in \fg(\K)$ and 
constructed by means of ${\mathfrak {sl}}_2$-triples. We need an explicit 
construction in Denef--Pas language,
so we write $\widetilde N_n$ for the $n \times n$ matrix with $1$'s on
the superdiagonal and $0$'s elsewhere, and then put
$N = \widetilde N_{2m}$ (if $\fg = \mf{sp}_{2m}$),
$N = \left[\begin{smallmatrix}
\widetilde N_m &   &                 \\
               & 0 &                 \\
               &   & -\widetilde N_m
\end{smallmatrix}\right]$ (if $\fg = \mf{so}_{2m + 1}$),
or
$N = \left[\begin{smallmatrix}
\widetilde N_m &                 \\
               & -\widetilde N_m
\end{smallmatrix}\right]$ (if $\fg = \mf{so}_{2m}$).
Note that, since we just need some constructible 
normalization in what follows, we do not need to deal with the question
\ins{of}
whether our construction depends on the choice of $N$: as long as we always 
make the same choice of $N$, we will get the same bijection between the points
 of $B_{\fg, r}(\rf_\K)$ and the equivalence classes of elements. 
 As explained in \cite{kottwitz:transfer-factors-lie}*{Section 2.4}, the value of the 
Kostant section 
map at $X$ is obtained by intersecting the stable orbit of $X$ with a certain subvariety
of $\fg$. We will refer to this subvariety as the Kostant section, as well. 

Let $X_0$ be the element of the intersection of our Kostant
section with the stable orbit that corresponds to the 
point $x$ (\ie, the stable orbit of $X$), and let 
$(\tilde I, \smash{\tilde I}^{\ast}, (\tilde a_i), (\tilde c_i))$ be the Waldspurger's datum 
attached to $X_0$.
Note that $\tilde I = I$ and $\smash{\tilde I}^\ast = I^\ast$---%
that is, that the $I$ and $I^*$ pieces of the data for $X$ and $X_0$
are the same---%
since $X$ and $X_0$ are stably conjugate.
We also may, and hence do, identify, for a fixed $i \in I$,
the fields $F_i^\#$ corresponding to $X$ and $X_0$,
and the algebras $F_i$ corresponding to $X$ and $X_0$,
see Proposition \ref{proposition:W&r-equivalence}(\ref{Wdatum3}).
%	\Lxxx{Julia will reference the right part of the proposition.}
For $i\in I^{\ast}$, let 
$\epsilon_i:=\sgn_{F_i/F_i^{\#}}(c_i{\tilde c}_i^{-1})$, where 
$\sgn_{F_i/F_i^\#}$ is the sign character of $F_i^\#$ 
(this definition makes sense, since the element $c_i\tilde c_i^{-1}$ is 
stable under the involution 
$\tau_i$ of $F_i$, and therefore lies in $F_i^{\#}$).  

Finally, let us define $\nu_{r,\K}(X)$ as
$$
\nu_{r,\K}(X)=\left(x,(\epsilon_i)_{i\in I^{\ast}(X)}\right)\in B_{\fg, r}(\rf_\K).
$$
%by \cite{cunningham-hales:good}*{Theorem 4.4}.

{\sl Part (1).}
We need to show that if $\nu_{r,\K}(X)=y=(x, (\epsilon_i))\in B_{\fg, r}(\rf_\K)$, then
\begin{equation}\label{eq: thickened}
\nu_{r,\K}^{-1}(y)={\mathcal O}_r(X).
\end{equation}
As above, we denote by $(I, I^{\ast}, F_i, (a_i), (c_i))$ 
the Waldspurger's datum corresponding to $X$.
Suppose that $Y \in \nu_{r,\K}^{-1}(y)$ has Waldspurger's datum
$(I', I^{\ast\,\prime}, F_i', (a_i'),  (c_i'))$.
By the definition of $\nu_{r,\K}$, we have that $[Y]_r = [X]_r$,
so, by Proposition \ref{proposition:W&r-equivalence},
we may, and do, assume that $I = I'$, $I^\ast = I^{\ast\,\prime}$,
and, for all $i \in I$,
$F_i = F_i'$ and
the elements $a_i$ and $a_i'$ have the same valuation and 
angular component. 
We claim that $\sgn_{F_i/F_i^\#}(c_i {c_i'}^{-1})=1$.

Let us prove this claim.
Let $(I, I^{\ast}, (\tilde a'_i), (\tilde c'_i))$ be the Waldspurger's 
datum that corresponds to the element $Y_0$ that is stably conjugate to $Y$ 
and lies in the Kostant section. It follows from the construction of the Kostant section (see, \eg \cite{kottwitz:transfer-factors-lie}*{Section 2.4}) 
combined with the proof of \cite{cunningham-hales:good}*{Theorem 4.6} that we can assume that 
$\sgn_{F_i/F_i^\#}(\tilde c_i^{-1}\tilde c'_i)=1$.  
Indeed, by the proof of \cite{cunningham-hales:good}*{Theorem 4.6}, 
for each element $X''$ in the stable orbit of $X$ there exists a unique element $Y''$ in the stable orbit of $Y$, such that $X''-Y''$ lies in $\fg(\K)_{r^+}$, and this is the element
that satisfies $\sgn_{F_i/F_i^\#}(c_i(X'')c_i(Y'')^{-1})=1$ (where
$c_i(X'')$, resp., $c_i(Y'')$ stands for the $c$-part of the Waldspurger data for these 
elements). In particular, there is exactly one such element for $X_0$. Since the Kostant section
is continuous by construction, this element must be the element $Y_0$ that 
also lies in the Kostant section. 
Finally, by    
definition of 
$\nu_{r,\K}$, we have
\[
\sgn_{F_i/F_i^\#}(c_i\smash{\tilde c}_i\inv)
= \sgn_{F_i/F_i^\#}(c_i'\smash{\smash{\tilde c}'}_i\inv),
\]
so that
\[
\sgn_{F_i/F_i^\#}(c_i c_i'{}\inv)
= \sgn_{F_i/F_i^\#}(\tilde c_i\smash{\smash{\tilde c}'}_i\inv)
= 1.
\]
%Since $\sgn_{F_i/F_i^\#}(\tilde c_i^{-1}{\tilde c'_i})=1$ it follows that 
%$\sgn_{F_i/F_i^\#}(c_i {c_i'}^{-1})=1$.
	%% Now incorporated into the equation.
	
%\Lxxx{This is$\sgn_{F_i/F_i^\#}(\tilde c_i\tilde c'_i^{-1})=1$ not clear to me.  By definition, we
%have, with the obvious notation, that
%$
%but why do we know that the right-hand side is $1$?

To summarize, we have shown that if $\nu_{r, \K}(Y)=\nu_{r, \K}(X)$, then
the elements $X$ and $Y$ have the data $(I, I^{\ast}, F_i, (a_i), (c_i))$  and 
$(I, I^{\ast}, F_i, (a_i'), (c_i'))$, respectively; and, for all $i \in I$,
the elements $a_i$ and $a_i'$ have the same valuation and 
angular component; and  $\sgn_{F_i/F_i^\#}(c_i {c_i'}^{-1})=1$.
It is shown in the proof of \cite{cunningham-hales:good}*{Theorem 4.6} that
these conditions are equivalent to the existence of topologically unipotent elements $u_i\in F_i$ such that $\lambda_i(X)=u_i\lambda_i(Y)$, where 
$\lambda_i$ are the appropriately labelled eigenvalues. This, in turn,
combined with the conditions   $\sgn\ins{_{F_i/F_i^\#}}(c_i {c_i'}^{-1})=1$, is equivalent 
to the statement that there exists a rational conjugate $Y'$ of $Y$, such that
$X-Y'$ lies in $\fg(\K)_{r^+}$ 
(note that we have shown in Remark \ref{rem:good-depth}
that the constant that is used to relate depth and slope in 
\cite{cunningham-hales:good} is, in fact, $1$ in all our cases). 
%	\Lxxx{How can this be?  So far, our conditions on
%$Y$ have pinned it down only in its conjugacy class.  Thus,
%I guess that this should be interpreted as saying that
%there's some element of the conjugacy class that is close to
%$X$}
%This implies the containment  
%$\nu_{r,\K}^{-1}(x)\subset {\mathcal O}_r(X)$.
%Note that the key component in this argument are the equations
%\cite{cunningham-hales:good}*{}, that can be rewritten as:
%The opposite containment is proved by reversing the argument.

Finally, note that for $x\in S_{\fg, r}(\rf_\K)$, 
upon taking the union over all $y\in B_{\fg, r}(\rf_\K)$  satisfying 
$\pi_r(y)=x$ on both sides of \eqref{eq: thickened}, 
we get the assertion of Lemma \ref{lemma: thick}. 

\bigskip

%The second statement follows directly 
%from \cite{cunningham-hales:good}, even though 
%it is not explicitly stated there in this form.  

%Suppose that $X, X' \in \fg(r,\K)$ satisfy
%$\nu_{r,\K}(X) = \nu_{r,\K}(X')$.  By Part (1),
%$[X]_r = [X']_r$.
%In this situation, Theorem 4.6 of
%\cite{cunningham-hales:good} shows that certain
%distributions associated to $X$ and $X'$
%(their \term{stable orbital integrals})
%take the same value at a special function $f$
%(the characteristic function of a suitable scaling of
%$\fg(\mc O_\K)$).
	%% Here, we need the scaling to account for the fact
	%% that our elements may have negative depth.
%The proof there involves first replacing $X'$ by a stable
%conjugate whose Waldspurger datum has a certain property,
%but the proof of Part (1) shows that $X'$ itself already has
%this property.
%It then appeals to Corollary 1.30 \loccit to show
%that the (rational) orbital integrals are already equal.
%Upon replacing that corollary (which is adapted to the
%particular function $f$) by
%Theorem 1.26 \loccit and
%the observation that
%$|D^{\fg, \mf l}(X)| = |D^{\fg, \mf l}(X')|$
%whenever \mf l is a Levi \K-subalgebra of \fg containing
%$C_\fg(X) = C_\fg(X')$
%(see \cite{cunningham-hales:good}*{(6.1.3)}), we obtain
%the desired equality.

{\sl Part (2).}
By Part (1), given a field $\K$ with sufficiently large residue characteristic,
we can unambiguously use the notation 
${\mathcal O}_r(y):=\nu_{r,\K}^{-1}(y)$, for $y\in B_{\fg, r}(\rf_\K)$. 
To prove that $\mathcal O_r(y)$ is definable (\ie, that it is obtained  
from a definable subassignment by specialization), we use the techniques from 
\cite{gordon-hales:transfer}.  
Recall the notation, and a few facts proved there that are relevant 
for the present situation.
\begin{enumerate}
\item Each $F_i^{\#}$ and $F_i$ is a $\K$-vector space.
As in \cite{waldspurger:nilpotent}, 
the elements $(c_i)$ determine a
pairing $q_W$ on the vector space
$W \ceq \bigoplus_{i \in I} F_i$ such that $(V, q_V)$ and
$(W, q_W)$ are isomorphic as quadratic spaces
(where $V$ is the quadratic space naturally associated
to our group $G$).
Let $\phi:W\to V$ be such an isomorphism, and let 
$\phi_{\ast}:\lop(W)\to \lop(V)$ be the isomorphism induced by $\phi$. 

Let $L_i:F_i\to\lop(F_i)$ be the linear map that takes $w\in F_i$ to
the linear operator that acts by left multiplication by $w$ on $F_i$.
Let $L:W\to \lop(W)$ be the direct sum of the maps $L_i$. 
Thus, if the space $(W, q_W)$ was constructed from Waldspurger's datum for 
a regular semisimple element $X$, then for $w\in W$, its image 
$\phi_{\ast}\circ L (w)$ can be thought of as an element of
$C%
%_{\fg(\K)}
%	%%I think.
(X)$%
\fto{ -- }{---}%
the centralizer of $X$.

\item Let $X$ be an $n\times n$-matrix with distinct eigenvalues 
$\{\lambda_j\}$, and with the 
characteristic polynomial $P_X(\lambda)=\prod_j(\lambda-\lambda_j)$.
Let $P^{(j)}(\lambda)$ be the polynomial $P_X$ divided by the factor  
$(\lambda-\lambda_j)$.
Then the projection operator onto  the $\lambda_j$-eigenspace can be 
written as $P^{(j)}(X)/P^{(j)}(\lambda_j)$.  More generally, 
if we write $P_X=f\tilde f$ as a product of two monic polynomials 
$f$ and $\tilde f$, then    
there exists a polynomial 
$\Pi = \Pi(\lambda, f, \tilde f)$ (see \cite{gordon-hales:transfer}*{\S1.7})
that, when evaluated  at $X$, defines the projection operator onto 
the direct sum of the eigenspaces of $X$ that correspond to the roots of 
$f$. The key fact here is that the coefficients of $\Pi$ are
polynomial expressions in the matrix entries $x_{ij}$ of $X$, and in the 
coefficients of $f$ and $\tilde f$.

%\item There is a condition ``trace-form'' (Definition 24) that states
%where $c$ is an element of $C(X)$ (a matrix variable)
\item \cite{gordon-hales:transfer}*{Lemma 35}
Let $z_i\in F_i^{\#}$ be an arbitrary element. Then the following conditions
are equivalent:
\begin{enumerate}
\item $\sgn_{F_i/F_i^{\#}}(z_i)=1$.
\item Let $w$ be an arbitrary element of $W$ such that $w_i=z_i$. Let $P$ be the projector from $C(X)$ onto $\phi_{\ast}\circ L(F_i)$ as above. (Essentially, it is a projector to the $\lambda_i$-eigenspace of $X$.)
Then there exists 
$X_1\in C(X)$ such that $PX_1\tau(X_1)=P(\phi_{\ast}\circ L)(w)$.
Here $\tau$ is the involution on $C(X)$ 
defined by $\tau(g)=(J^{-1}){}^{t}gJ$, where $J$ is the matrix from the 
definition of the group $G$ (see
\eqref{eq:J-symp} and \eqref{eq:J-orth}
on page \pageref{eq:J-symp}).
\end{enumerate}
\end{enumerate}

Now we are ready to finish the proof that the sets $\mathcal O_r(y)$ are 
definable for $y\in B_{\fg, r}(\rf_\K)$, and depend on $y$ in a definable way.
Let $y\in B_{\fg, r}(\rf_\K)$. 
Then $y=(x, (\epsilon_i)_{i\in I^{\ast}(x)})$, 
where $x\in S_{\fg, r}(\rf_\K)$ and $\epsilon_i=\pm 1$.

By the construction of the subassignment $B_{\fg, r}$, the number
of indices $i$ such
that $\epsilon_i = 1$ equals the cardinality $|I^{\ast}(x)|$ 
by Lemma \ref{lem:cardinality}. 
Let $X\in {\mathcal O}_r(y)$.
We have the element $X_0$ in the stable conjugacy class of $X$ that lies 
in the Kostant section. We can think of the entries of the matrix $X_0$ 
as terms in Denef--Pas language (see \cite{CHL}).  
Now, let $(\tilde c_{i})$ be the $(c_i)$-part of the Waldspurger's datum that 
corresponds to $X_0$, and let $c_0=\phi_{\ast}\circ L (\tilde c_{i})$.  Note 
that all of this can be written as one logical formula in the notation of 
\cite{gordon-hales:transfer}:
$$\exists c_0\in C(X_0) \quad \text{trace-form}(X_0,c_0)$$
(see Definition 24 \loccit for the definition of the term ``trace-form'', and 
Remark 34 \loccit for this statement).
Similarly, for our element $X$, we have an element 
$c\in C(X)$ satisfying the same condition with $X$ replacing $X_0$.
The condition that $X\in {\mathcal O_r(y)}$ can be restated as 
$$
\sgn_{F_i/F_i^{\#}}(c_i\tilde c_{i}^{-1})=\epsilon_i\ \forall i \in I^\ast.
$$
By Lemma 35 \loccit\ \out{quoted above}
(\fto{see}{quoted as} item (3)\ins{ above}), this condition can be 
expressed by a formula in Denef--Pas language, whose variables include 
$x$ (interpreted as a tuple of variables of the residue field), and 
$\epsilon_i$. In particular, the dependence on $y$ is definable. 

%First, we need to prove that this map is well-defined as a map of $r$-equivalence classes, that is,
%if $[X]_r=[X']_r$, then $\mu(X)=\mu(X')$.
%This follows from 
%Suppose an element $X'$ maps to the same point in $B_{\fg, r}(\rf_\K)$ as $X$.

\bigskip 
{\sl Part (3)}. 
This follows immediately from Part (1), combined 
with Proposition \ref{prop: separates}.

\bigskip

{\sl Part (4).}
Finally, to prove the last statement, we just need to modify slightly 
the argument
of \cite{cunningham-hales:good}*{Lemma 6.2}. 
For a point $y\in B_{\fg, r}(\rf_\K)$, let $\mathcal O_r(y)=\nu_{r,\K}^{-1}(y)$, 
as above.
It is an open subset of $\fg(r, \K)$. 
Let $X_y$ be an element in $\nu_{r,\K}^{-1}(y)$, and let 
$\ft$ be the Cartan subalgebra containing $X_y$.  
We think of $\mathcal O_r(y)$ 
as a neighbourhood of the orbit of $X_y$, and recall that 
$$\mathcal O_r(y)=\cup_{Y\in \ft(\K)_{r^+}}\mathcal O(X_y+Y).$$ 

Recall the Weyl integration formula:
% as stated in \cite{waldspurger:loc-trace-form}:
\begin{equation}\label{eq:weyl}
\begin{aligned}
\int_{\fg(\K)} f(X)dX
&=\sum_{T'} |W(G(\K), T'(\K))|^{-1}
\int_{\ft'(\K)}|D^{\fg}(X)|\int_{G(\K)/T'(\K)} f(\Ad x X) dx\ins\,dX,
\end{aligned}
\end{equation}
where the summation is over the $G(\K)$-conjugacy classes of maximal tori 
$T'$
in $G$. 
Let
%$f\in {\mathcal H}_r$
$f \in \Hecke_r(\fg(\K))$
be a definable test function (by which we mean that it is a specialization of 
some constructible motivic function), and let
$f_y$ be the 
characteristic function of the set $\mathcal O_r(y)$.
Let us apply the Weyl integration formula to the function $f f_y$.
It follows from the proof of \cite{cunningham-hales:good}*{Theorem 4.6}
that, if two restricted elements $X, X'\in \fg(r, \K)$ satisfy 
$\nu_{r,\K}(X)=\nu_{r,\K}(X')$, then their centralizers are $G(\K)$-conjugate, 
{\it cf.} Proposition \ref{proposition:W&r-equivalence}(\ref{Wdatum1}). 
Therefore, since $f_y$ vanishes outside $\mathcal O_r(y)$, the right-hand side
\ins{of \eqref{eq:weyl}}
will
have only one non-\fto{$0$}{zero} summand, corresponding to $T=C_G(X_y)$. 
Then
the right-hand side equals
$$
|W(G(\K), T(\K))|^{-1}\,\int_{X_y+\ft_{r^+}(\K)}|D^{\fg}(X)|\,
\int_{G(\K)/T(\K)}f(\Ad_x X)\,dx \,dX.
$$
Since the orbital integral
 $$\int_{G(\K)/T(\K)}f(\Ad_x X)\,dx$$ 
as a function of $X$, is constant on the coset $X_y+\ft_{r^+}(\K)$
by Proposition \ref{prop: separates},  we get:
\begin{equation}\label{eq:orbital integral}
\int_{\fg(\K)} ff_y(X)dX
=|W(G(\K), T(\K))|^{-1}|D^{\fg}(X_y)|\mu_{X_y}(f)\vol^{\ast}(X_y+\ft_{r^+}(\K)),
\end{equation}
where $\vol^{\ast}$ is the volume on $\ft(\K)$, (see 
\cite{cunningham-hales:good}*{Section 6} for the discussion of the appropriate normalization for $\vol^{\ast}$). 
The factor 
$|W(G(\K), T(\K))|^{-1}\,|D^{\fg}(X_y)|$ is constructible by 
\cite{cunningham-hales:good}*{Lemma 6.1 and (6.1.3)}. 
The factor $\vol^{\ast}(X_y+\ft_{r^+}(\K))$ is a specialization of a 
constructible motivic function of $y$ by \cite{CL}*{Theorem 10.1.1}, since 
the set $X_y+\ft_{r^+}(\K)$ depends on $y$ in a definable way. 
We denote by $Q_{\fg, r}$ the element of ${\mathcal C}(B_{\fg, r})$ 
that specializes to
$$
|W(G(\K), T(\K))|^{-1}\,|D^{\fg}(X_y)|\,\vol^{\ast}(X_y+\ft_{r^+}(\K)).
$$

The left-hand side of (\ref{eq:orbital integral}) 
is $\int_{\mathcal O_r(y)}f(X) dX$. 
By Part (3), the set $\mathcal O_r(y)$ is definable 
and depends on $y$ in a definable way. 
More precisely, it is a specialization of the fibre over $y$ of an   
an element of $\de_{B_{\fg, r}}$; therefore, the left-hand side of 
(\ref{eq:orbital integral}) is a specialization of a constructible 
motivic function of $y$, by \cite{CL}*{Theorem 10.1.1}.

The last claim now
follows from \cite{CL}*{Theorem 10.1.1} 
together with the results on 
specialization of motivic integrals,
\cite{CL.expo}*{Theorems 6.9 and 7.3}.
\end{proof}

%\begin{corollary}\label{cor: separates}
%Suppose $X, X'\in \fg(r,\K)$. Then 
%\[
%\nu_{r,\K}(X) = \nu_{r,\K}(X') \iff \forall f\in \Hecke_r(\fg(\K)),\ins\ \mu_X(f) = \mu_{X'}(f).
%\]
%\end{corollary}
%\begin{proof}
%Combine Proposition~\ref{prop: B_g,r}(1) with Proposition~\ref{prop: separates}.
%\end{proof}

\subsection{Proof of Theorem \ref{thm:distribution characters}}\label{sub: proof thm 2}

\subsubsection{Part (1)} 
Let $\K$ be any local field of residue characteristic bigger than $2$,
$\pi$ a restricted representation of $G(\K)$
of minimal depth $r$,
and
$\Psi$ the associated cuspidal datum (see Section \ref{sub: JKdatum}). 

Recall
from Theorem \ref{thm:main-ext-ingredient}
that there is some element $\Gamma_\Psi \in \fg({-r},\K)$ such
that, for every test function $f$ with support contained in $G(\K)_r\reg$, 
we have 
\begin{equation}\label{eq:FT}
\frac{1}{\deg(\pi)}\Theta_{\pi}(f)=\hat{\mu}_{\Gamma_\Psi}(f\circ \mexp^{-1})=
\mu_{\Gamma_\Psi}(\widehat{f\circ \mexp^{-1}}).
\end{equation}
By Lemma \ref{lemma:  Hecke_r},
$\widehat{f\circ \mexp^{-1}}$ lies in the space
$\Hecke_{-r}(\fg(\K))$.
By Proposition \ref{prop: B_g,r}(1), there exists $M>0$ such that 
if the residue characteristic of $\K$ is greater than $M$,
the element $\Gamma_{\Psi}$ projects to a point $y_{\Psi}\in B_{\fg, -r}(\rf_\K)$.  
By Proposition \ref{prop: B_g,r}(3), the restriction of the 
distribution $\mu_{\Gamma_{\Psi}}$ to the space $\Hecke_{-r}(\fg(\K))$
depends only on the point $y_{\Psi}$.
Hence, the point $y_{\Psi}\in B_{\fg, -r}(\rf_\K)$ determines the  
restriction of $\Theta_{\pi}$ to $G(\K)\reg_r$. 
This argument shows that for fields $\K$ with sufficiently large residue 
characteristic, the map 
$$\pi_\Psi\mapsto \nu_{-r,\K}(\Gamma_{\Psi})$$ is a well defined one-to-one  
map from the set of 
$r$-equivalence classes of restricted representations of $G(\K)$
of minimal depth $r$
to
$B_{\fg, -r}(\rf_\K)$.

\subsubsection{Part\ins{ }(2)}
This follows immediately from the equality \eqref{eq:FT} and 
Proposition \ref{prop: B_g,r}(4).

\subsubsection{Part (3)}
We have proved that the sets ${\mathcal O}^{-r}(y)$ form a definable family of definable sets parameterized by a variable  $y$ running over the definable 
subassignment $B_{\fg, -r}$. Therefore, given a definable family of definable functions $\{f_a\}$ (parameterized by $a$ in some subassignment $A$), 
the expression  
$F(y,a):=\displaystyle\int_{\mathcal O^{-r}(y)}\widehat{f_a\circ \mexp^{-1}}$ 
is a constructible motivic exponential function of $y$ and $a$  by  
\cite{CLF}*{Theorem 4.1.1}.
Then, by the specialization principle \cite{CLF}*{Theorem 9.1.5}, there exists a constant $M'$ (that depends on the family $\{f_a\}$ and on $\fg$ and $r$), such that for all fields $\K\in {\mathcal F}_{\Z,M'}$, the motivic integral
$\displaystyle\int_{\mathcal O^{-r}(y)}\widehat{f_a\circ \mexp^{-1}}$ 
specializes to the $p$-adic integral 
$\displaystyle\int_{\mathcal O^{-r}_{y, \K}}\widehat{f_{a, \K}\circ \mexp^{-1}}$. 

Further, let $M''$ be the constant such that for all 
$\K\in {\mathcal F}_{\Z, M''}$, the specialization principle holds for
the constructible motivic function  $Q_{\fg, -r}\in {\mathcal C}(B_{\fg, -r})$
from Propositon \ref{prop: B_g,r}, Part (\ref{item:Q}).
Note that $M''$ depends only on $\fg$ and $r$.

%The expression 
%$$|W(G(\K), T(\K))|\,|D^{\fg}(X_{y})|^{-1}\vol^{\ast}(X_y+\ft_{r^+}(\K))$$
%does not depend on $a$ and 
%is a
%constructible function of $y$, see \cite{cunningham-hales:good}).
%	\Lxxx{We don't have to worry about the fact that we're
%\emph{inverting} the measure?  (Elsewhere, we make some effort
%to take care of this.)}
%More precisely,
%there exists an element of that specializes to it -- 

Let 
$M_0$ be the maximum of the 
 constant $M$ from Part (1) and $M'$, $M''$ defined above.
Let $\K\in {\mathcal F}_{\Z,M_0}$, and 
let $\pi$ be a restricted representation of $G(\K)$ of minimal depth $r$,
as above.
%\Lxxx{What's $M$?}
Let
$\Gamma\in \fg(-r, \K)$ be the element corresponding to $\pi$
(see Definition \ref{defn:X*_Psi}\ins), and  let 
$y\in B_{\fg, -r}(\rf_{\K})$ be the image of $\Gamma$ under $\nu_{r, \K}$.
Then, using the 
notation of the proof of Proposition \ref{prop: B_g,r} we can write 
$\Gamma=X_y$.  Let $T$ be the centralizer of $X_y$.
We have (using the equality (\ref{eq:orbital integral})):
\begin{equation*}
%\begin{aligned}
\frac{1}{\deg(\pi)}\Theta_{\pi}(f_a)
=\mu_{\Gamma_{\Psi}}(\widehat{f_a\circ \mexp^{-1}})
%=\frac{|W(G(\K), T(\K))|}{|\,D^{\fg}(X_{y})|\vol^{\ast}(X_y+\ft_{r^+}(\K))}
%\int_{\mathcal O_{-r}(X_y)}\widehat{f_a\circ \mexp^{-1}}=
\frac{1}{Q_{\fg, -r,\K}(y)}F_{\K}(y,a).
%\end{aligned}
\end{equation*}
This, we have shown that 
$\frac{Q_{\fg, -r, \K}(y)}{\deg(\pi)}\Theta_{\pi}(f_a)$
is a constructible exponential function of $a$. Since 
$\frac{Q_{\fg, -r, \K}(y)}{\deg(\pi)}$ is a positive rational number that does not depend on $a$, this completes the proof.
\qed

\begin{remark}
%Reconstructing the distribution character from a point 
%$x\in B_{\fg, -r}(\rf_\K)$
Rephrasing Part (1) of the theorem we just proved, we would like to emphasize 
that we can reconstruct (in an algorithmically computable way)
the values of the 
distribution character, at least at a family of definable test 
functions supported near the identity and satisfying the conditions 
from Definition \ref{def:definable equivalence}, 
from a point $y\in B_{\fg, -r}(\rf_\K)$.
Indeed, starting with a point $y\in B_{\fg, -r}(\rf_\K)$, we can construct the fibre 
${\mathcal O}^{-r}_{y,\K}$ of the subassignment 
${\mathcal O}^{-r}$ from Proposition \ref{prop: B_g,r}(3), 
for all local fields \K with residue characteristic bigger than a constant, 
which we will denote by $M_{\fg, r}$, that depends only on $\fg$ and $r$. 
Now, let us consider the family $\xi_{a,n}$ of Section \ref{subsub:thefamily}. By Lemma 
\ref{lem:thefunctions}, there exists a constant $M_{\xi,\fg,r}>0$, such that for 
$\K\in {\mathcal F}_{\Z, M_{\xi,\fg,r}}$, 
  the motivic integral $\displaystyle\int_{\mathcal O^{-r}}\xi_{a,n}$
specializes to the integral 
$\displaystyle\int_{\mathcal O^{-r}_{y, \K}}\xi_{a,n, \K}\,dg$, where the latter integral is with respect to 
the Haar measure on $G(\K)$. Here we apply Lemma \ref{lem:thefunctions} to the family 
$\{f_{\alpha}\}$ that consists of 
the characteristic functions of the thickened orbits $\mathcal O^{-r}_{y, \K}$, 
indexed by 
the points $(y, \K)$ of the definable subassignment $B_{\fg, -r}$.
The constant $M_{\xi, \fg, r}$ depends on $M_{\xi}$ and on the formulas defining the subassignment 
$B_{\fg, -r}$, which, in turn, depend only on $G$ and $r$. For the fields of characteristic zero, where we know that the distribution character is a continuous distribution, this information is equivalent to knowing the restriction of the
distribution character to $G(\K)_r$. 
\end{remark}

\subsection{Proof of Theorem \ref{thm:character} in the 
characteristic-zero case}\label{sub: proof thm 3}

We will use J.~Korman's theorem on the local constancy of characters 
\cite{korman:local-constancy}, which
we quote here.  That paper assumes that characteristic of $\K$ is zero,
so we will assume that throughout this section.
%	\Lxxx{I \emph{think} that we're imposing this
%characteristic-$0$ assumption, but we don't mention it
%explicitly.}

Let $\gamma\in G(\K)\reg$ be a {\it compact} element,
and assume that the connected component of the  centralizer of $\gamma$ is a torus that splits over a 
tamely ramified extension $E$ of $\K$. 

Let $T=C_G(\gamma)^{\circ}$.
%\Jxxx{I am not at all sure this is 
%correct. Will think more about what $T$ is.}   
Then the \term{regular depth} of $\gamma$ 
(see \cite{korman:local-constancy}*{Definition 1.1}
or \cite{adler-korman:loc-char-exp}*{Definition 4.1}, where
the term \emph{singular depth} is used instead)
is the rational number 
$$
s(\gamma)=\max\{\ord_\K(\alpha(\gamma)-1)\mid \alpha\in {\Phi}(G, T)\}.
$$
Note that the following theorem does not require
our strong hypotheses on the representation $\pi$.

\begin{theorem}[\cite{korman:local-constancy}*{Theorem 4.1}]
\label{thm:korman}
Let
$\pi$ be an irreducible admissible representation and choose
$r'>\max\{s(\gamma),\depth(\pi)\}$, where $\depth(\pi)$ is the depth of the 
representation $\pi$. Then 
the distribution $\Theta_{\pi}$ is represented on the set 
$\lsup{G(\K)}\bigl(\gamma T(\K)_{r'+s(\gamma)}\bigr)$ by a constant function.
\end{theorem}

In fact, this theorem follows from Corollary 11.9 of
\cite{adler-korman:loc-char-exp} (since, in the notation of
that result, we have $\mf m = \Lie(T)(\K)$ in our case, so that
$\mc O_{\mf m}(0) = \sset0$),
which makes no asumption on 
the characteristic of the field $\K$.
Hence, the agument presented in this section  for the fields of characteristic zero works 
in positive characteristic as well. 
However, since part of the strength of motivic integration lies in the transfer between 
characteristic-zero and positive-characteristic cases, we decided to include the next section, in which we derive the positive-characteristic version of Theorem \ref{thm:character} from the 
characteristic-zero case.

%, except that that result is
%actually slightly better---it requires (also in the notation of
%that result) only $r > \max\{\rho(\pi), 2s(\gamma)\}$.  (This
%is actually strictly \emph{better} in general, since
%Jeff--Jonathan's $r$ is playing the r\^ole of our
%$r' + s(\gamma)$; but it doesn't make any difference here,
%because we have that
%$s(\gamma) \ge \depth(\gamma) \ge \depth(\pi)$
%in the range where we can say anything.)
%Have I missed some other part of this section that depends
%on being in characteristic $0$?  If I have not, then I
%strongly feel that we should, at the very least, cite
%\cite{adler-korman:loc-char-exp} here in addition to
%\cite{korman:local-constancy}.
%I'm not averse to explaining how our results make it easier
%to transfer results from characteristic $0$ to positive
%characteristic (and \textit{vice versa}), but I think it is
%very important to give Jeff's work due credit.}

%Note that the field $\K$ is assumed to have characteristic zero everywhere 
%in \cite{korman:local-constancy}, and this forces us to restrict our 
%arguments to the fields of characteristic zero, for now. 

Now let us assume that the residue characteristic $p$ is larger than
$n = n_G$, in the notation of \S\ref{sec:hyps},
and that $\gamma \in G(\K)_r\reg$%
%	%% It's necessary to bind $\gamma$ here.  We
%	%% used to (sort of) bind it below, but that meant we used
%	%% it before binding it!
.
Then, since $G$ embeds in $\GL_n$ and
(in this setting) every torus in $\GL_n$ splits over a tame
extension of \K,
the condition that the 
centralizer of $\gamma$ splits over a tamely ramified extension is 
fulfilled automatically.
Since $\gamma$ is compact, we can apply
Theorem \ref{thm:korman}.

Let $G_{0^+}\reg$ be the subassignment that specializes, for every 
$\K\in {\mathcal F}_{\Z, M}$, to the set of regular topologically unipotent 
elements $G(\K)_{0^{+}}\reg$ (see \S\ref{subsub:definable}).
Thus, the term ``definable function on $G(\K)_{0^+}\reg$'' really
means ``specialization of a definable function on $G_{0^+}\reg$''.
%	\Lxxx{I added the previous sentence because, otherwise,
%it seemed that we defined $G_{0^+}\reg$ a full section before
%we used it!
%	Also, is it possible that we mean `constructible'
%rather than `definable'?  I can't seem to find where we defined
%the latter term for functions (although we do seem to use it a
%few times as a synonym for `constructible'). 
%Julia says: I've added the defn of definable function, and removed the last sentence before this block. OK? If not, let's try to talk rather than correspond inside the paper :) }

\begin{lemma}\label{lem:s_gamma}
There exists $M>0$, and a definable function $\phi(\gamma)$ on the 
set $G(\K)_{0^+}\reg=G_{0^+, \K}\reg$ that satisfies
$\phi(\gamma)\ge s(\gamma)$ for all $\gamma\in G(\K)_{0^+}\reg$, for all 
$\K\in {\mathcal F}_{\Z, M}$. 
\end{lemma}

\begin{proof}
Let $X=\mexp^{-1}(\gamma)$, a regular\ins, semisimple\ins, topologically nilpotent element in $\fg(\K)$. 
Then 
$s(\gamma)=\max\{\ord_\K(d\alpha(X))\mid \alpha\in {\Phi}(G, T)\}$. 
Note that $D^{\fg}(X)=\det \ad(X)$ is obviously a definable function of $X$, and it has the expression 
$$D^{\fg}(X)=\det\ad(X)=\prod_{\alpha}d\alpha(X).$$ 
Therefore, since $d\alpha(X) \in \mc O_\K$ for every
$\alpha \in \Phi(G, T)$, we have for all such $\alpha$ that
$$
\ord_\K(d\alpha(X))\le 
\ord_\K D^{\fg}(X);
$$
so we can take $\phi(\gamma)=\ord_\K D^{\fg}(X)$.
\end{proof}

With some care to handle the case when $T$ is not
split, we expect that we could show that $s(\gamma)$ itself is a
definable function.  However, for our purposes, this is not
necessary.

\begin{corollary}\label{cor:T_gamma}
%\begin{enumerate}
%\item 

 There exists $M>0$ and a definable subassignment ${\mathcal T}$ in the 
category  $\de_{G_{0^+}\reg}$, such that for every 
$\K\in {\mathcal F}_{\Z, M}$, and every $\gamma\in G(\K)_{0^+}\reg$, the fibre 
${\mathcal T}_{\gamma, \K}$
is the set
$\gamma T(\K)_{r+2\phi(\gamma)}$. 
The restriction of the character
$\theta_{\pi}$ to each of the sets 
$\gamma T(\K)_{r+2\phi(\gamma)}$ is constant.
\end{corollary}

\begin{remark}
We observe that the corollary implies that the sets 
$\lsup{G(\K)}\bigl(\gamma T(\K)_{r+2\phi(\gamma)}\bigr)$
can also be thought of as the fibres of a definable subassignment.
However, we cannot tell whether the individual level sets 
$\lsup{G(\K)}\bigl(\gamma T(\K)_{r'+s(\gamma)}\bigr)$ are definable 
(\ie, whether they are in $\de_{\Q}$)---%
\textit{a priori}, there is no reason why we could eliminate the parameter $\gamma$, 
so this corollary does not automatically imply that $\theta_{\pi}$ is 
a constructible exponential function.  
%	\Lxxx{Do we mean $T(\K)_{r + 2\phi(\gamma)}$ instead
%of $T(\K)_{r' + s(\gamma)}$? Julia says: No, we don't. I've changed the wording above slightly to make the meaning clearer.}
\end{remark}

It follows from the above corollary that  $\lsup{G(\K)}{\mathcal T}_{\gamma, \K}$ is a definable neighbourhood of the conjugacy class 
of $\gamma$ on which the character $\theta_\pi$ is constant.
We would like to get a compact definable neighbourhood of $\gamma$, with 
the additional property that its volume is an invertible element in the 
ring of values of the motivic measure.  First, we replace the neighbourhood
$\lsup{G(\K)}{\mathcal T}_{\gamma, \K}$ by something compact.
Let $x$ be the hyperspecial vertex in $\BB(G, \K)$ such that
$G(\K)_{x, 0}=G({\mathcal O}_\K)$---which is, clearly, definable. 
Then by the above corollary, 
%	\Lxxx{I don't see how this follows from the corollary.  (I don't
%quite understand, either, how the corollary follows from the lemma,
%so maybe I'm missing something; but it seems to me that what's really
%going on here is that we're using the same reasoning as was used in
%the corollary.)
%	For that matter, do we really need the corollary, since we
%immediately throw away the subassignment $\mathcal T$, and never
%use it again?
%Julia says: let's talk about it. I stand by what's written, for now.}
there is a definable subassignment, which we will denote by ${\mathcal T}^0$,
such that 
\begin{equation*}
{\mathcal T}^0_{\gamma, \K}=
\lsup{G(\K)_{x, 0}}\bigl(\gamma T(\K)_{r+2\phi(\gamma)}\bigr)
\end{equation*} 
for $\K$ of sufficiently large residue characteristic and $\gamma\in G(\K)$.
The set ${\mathcal T}^0_{\gamma, \K}$ is a compact, definable neighbourhood of 
$\gamma$ on which $\theta_{\pi}$ is constant.
We have that $\vol\bigl({\mathcal T}^0_{\gamma, \K}\bigr)$ is a 
constructible motivic function of $\gamma$, since it computes the volumes of 
the fibres of a definable subassignment.
However, later we would like to divide another constructible motivic function 
by this volume function, and it might not be an invertible element in the ring of 
constructible motivic functions.  The next lemma shows how to
fix this by passing to a further subassignment.
%	\Lxxx{I changed the wording here, since it seemed that we were
%basically stating the lemma twice!}
%that we can make a 
%smaller  definable subassignment $U$, such that each fibre $U_{\gamma, \K}$ is contained in the corresponding fibre  ${{\mathcal T}^0}_{\gamma, \K}$, and such 
%that the motivic volumes of the fibres $U_{\gamma}$  are invertible in the ring 
%$K_0(\rde_{\Q})\otimes A$
%where
%motivic volume takes values.
%	\Lxxx{Don't we really mean that the volume
%\emph{function} $\gamma \mapsto U_{\gamma, \K}$ is
%invertible (as we said above)?  Volumes of non-empty, open subsets
%are, of course, \emph{always} invertible as numbers.}
 
%  function $v(\gamma):=
%\vol$
%on $G(\K)_{0^+}\reg$
%is constructible.

%\end{enumerate}
%\end{corollary}
%	\Lxxx{We run into problems here, since the $G(\K)$-orbit of
%$\gamma T(\K)_{r + 2\phi(\gamma)}$ has infinite volume
%(unless $G(\K)$ is compact).  Since I'm not sure how
%sensitive definability is to passing to an orbit, I don't
%know if it's OK to replace the $G(\K)$-orbit by the
%$G(\K)_{x, 0^+}$ orbit (for some $x \in \BB(T, \K)$); but,
%if that works, then we're set.}

\begin{lemma}\label{lem:U_gamma}
There exists a definable subassignment 
$U$ in $\de_{G_{0^+}\reg}$, such that for every $\K\in {\mathcal F}_{\Z, M}$, 
for every $\gamma\in {G(\K)_{0^+}\reg}$,
we have $U_{\gamma, \K}\subset  {\mathcal T}^0_{\gamma, \K}$, and
the motivic volumes of the fibres $U_\gamma$ are of the form 
$\lef^{n(\gamma)}$, with $n(\gamma)$ a $\Z$-valued constructible function 
on  $G_{0^+}\reg$.
\end{lemma}

\begin{proof}
Note that ${\mathcal T}^0_{\gamma, \K}$ is an open subset of $G(\K)$ 
that depends only on the parameter $\gamma$.
Let $d$ stand for the dimension of $G$ (as an algebraic variety---which coincides with its dimension as a definable subassignment in the sense of 
\cite{CL}*{Section 3.1}, as well as with the dimension of $G(\K)$ as a 
$p$-adic manifold).
We will call a subset of $\A^d(\K)$ a \term{ball} if it is defined by a formula
$\ord(x_i-a_i)\ge r_i$ in the variables $(x_1, \dots, x_d)$ for some element 
$a=(a_i)_{1\le i\le d}\in \A^d(\K)$, and  for $r=(r_i)_{1\le i \le d}\in\Z^d$.  
Notice that the 
volume of such a $p$-adic ball with respect to the Haar measure on 
$\A^d(\K)$,
normalized so that $\meas \mc O_\K^d = 1$,
is of the form  $q^n$, where the power $n \in \Z$ depends on $r$.
	%% Deleted remark about normalization of measure,
	%% because we want $n$ to be an integer and that
	%% fails for some normalizations of measure.
Now, 
the idea is to make a subassignment $U$ such that $U_{\gamma,\K}$ is a  
ball (with respect to some suitable local coordinates around $\gamma$). 
 
One convenient choice of coordinates on $G$ for the study of Haar measure comes
from a big Bruhat cell. Let $T\textsub{spl}$ be a maximal \K-split torus in $G$, $B$ a 
Borel subgroup containing $T\textsub{spl}$, $N$ its unipotent radical,
$B^-$ the opposite 
Borel subgroup (with respect to $T\textsub{spl}$),
and $N^-$ the opposite unipotent radical $N^-$.
The subvariety $\Omega \ceq N^- T\textsub{spl}N$ of $G$,
which is a translate of the standard big cell in $G$,
is isomorphic to the Cartesian product
$N^- \times T\textsub{spl} \times N$.
The complement of $\Omega(\K)$ in  $G(\K)$ has Haar measure $0$. 
As varieties, $N$ and $N^-$ are isomorphic to affine   
spaces of the same dimension,
and $T$ to an open subset of an affine space.
Let $x_1,\dotsc, x_s$ be coordinates on $N$;
$y_1, \dotsc, y_s$ coordinates on $N^-$;
and
$\fto t s_1,\dotsc, \fto t s_r$ coordinates on $T\textsub{spl}$.
With this choice of coordinates,
the extension by $0$ to $G(\K)$ of the volume form
$$\omega=dy_1\wedge \dots \wedge dy_s\wedge dx_1\dots \wedge dx_s\wedge
\frac{ds_1}{s_1}\wedge \dots \wedge\frac{ds_r}{s_r}$$ 
on $\Omega(\K)$ is a Haar measure.
To abbreviate, we will write $(x,y,s)$ for the coordinates, meaning the 
corresponding tuples.

For every 
$\gamma\in  G(K)_{0^+}\reg\cap\Omega(\K)$, there exists a ball, 
(with respect to 
the coordinates $(x,y,s)$, in the sense of the definition above) 
that contains $\gamma$, and is contained in
\label{left-fibre}
$\Omega(\K)\cap {\mathcal T}^0_{\gamma, \K}$, and such that
$|s_i|$ is constant on this ball for $i=1,\dotsc, r$.
For every $\gamma$, there exists unique largest ball with these properties.
We denote it by $C_1(\gamma)$.
By the definition of the volume form $\omega$, the volume of this ball 
depends on $\gamma$ in a 
definable way, and is of the required form.
We define the subassignment $U$ of $G_{0^+}\reg\times \A^d$ in such a way that
$U_{\gamma,\K}=C_1(\gamma)$ for $\gamma\in  G(K)_{0^+}\reg\cap\Omega(\K)$.

Finally, there are finitely many translates of the set $\Omega(\K)$ 
that cover $G(\K)$.
%	%% This is true because the (translate of) the big cell
%	%% is (Zariski) open in the projective variety $G/B$,
%	%% so that finitely many translates cover $G/B$.
%	%% Julia says that anything that works geometrically
%	%% works definably, too (my paraphrase!).
Let us denote them by $\Omega_1, \dotsc, \Omega_l$, 
with $\Omega_1=\Omega(\K)$ (where $l$ depends only on $G$). 
These translates can be chosen to be definable.
%	%% Julia says:  Because we're translating a Zariski-open
%	%% set.
For the elements $\gamma$ in  
$\Omega_{j}(\K)\setminus \cup_{i=0}^{j-1}\Omega_i(\K)$, 
we define the ball $U_{\gamma,\K}$ in the same way, 
except using the coordinates on $\Omega_j$.

%Let $h_i\in G(\K)$ be an element such that $\Omega_i=h_i\Omega_1$.  
%We can pick some open compact definable subsets of 
%$\Omega_i^0\subset \Omega_i$ , $i=1,\dots, l$ such that 
%$G(\K)=\sqcup_{1\le i\le l}\Omega_i^0$ is a disjoint union.  
%For $1\le i\le l$, and 
%$\gamma\in \Omega_i\cap G(\K)_{0^+}\reg$, let $C_i(\gamma)=C_1(h_i^{-1}\gamma)$.
%	\Lxxx{Julia, I think that we decided to discuss this later,
%but forgot to put it on the list of questions.
%By definition, $C_1(h_i^{-1}\gamma)$ contains $h_i^{-1}\gamma$,
%but I don't think that it needs to contain $\gamma$.
%Thus, we should work with $h_i C_1(h_i^{-1}\gamma)$; but then
%it's not clear to me that this is still contained in
%$\mathcal T^0_{\gamma, \K}$.}
%Finally, let $C(\gamma)=\bigcap_{\Omega_i\ni \gamma}C_i(\gamma)$. 
%This is a ``ball'' around $\gamma$. 

\end{proof}

\subsubsection{Completing the proof}\label{subsub:end of proof}
Now we are ready to prove that there is a constructible
motivic exponential function on $G_{0^+}\reg$ whose
specialization has the same restriction to
$G(\K)_r\reg$ as the function $\theta_{\pi}$.
Let $r$ be the depth of $\pi$, as before (and recall that we are assuming that
it is the essential depth of $\pi$ as well).
Let $f_{\gamma}(x)$ be the characteristic function of 
$U_{\gamma, \K}$. Then by the local constancy result discussed above, 
\begin{equation}\label{eq:thetas}
\Theta_{\pi}(f_{\gamma})
=\vol\bigl(U_{\gamma, \K})\theta_{\pi}(\gamma).
\end{equation}

By Theorem \ref{thm:distribution characters}(2, 3), the functions 
$\frac{Q_{\fg, -r}(x)}{\deg(x)}\Theta_{x}(f_{\gamma})$ form a constructible family of
constructible motivic
exponential functions of $\gamma$, indexed by $x\in B_{\fg, -r}$.
Then, by Lemma \ref{lem:s_gamma} combined with \eqref{eq:thetas}, the functions 
 $$
\frac{Q_{\fg, -r}(x)}{\deg(x)}\theta_{x}(\gamma)=
q^{-n(\gamma)}\frac{Q_{\fg, -r}(x)}{\deg(x)}\Theta_{x}(f_{\gamma})
$$ 
form a constructible family of constructible motivic exponential functions,
and we can define 
\begin{equation*}
F(x, \gamma):=q^{-n(\gamma)}\frac{Q_{\fg, -r}(x)}{\deg(x)}\Theta_{x}(f_{\gamma}),
\end{equation*} 
which completes the proof.
We observe that dividing by the formal degree
was only necessary
to show that $F(x, \gamma)$ is constructible in the variable 
$x\in B_{\fg, -r}$
(see Remark \ref{remark:dep-on-haar}).

%%%%%%%%%%%%%%%%%end of Loren's piece %%%%%%%%%%%
%\begin{remark} --- I think it's best to remove it. J.
%We observe that by now we have two approaches to the character:  
%we can study the distribution character via testing in on a 
%definable family of definable test functions; 
%or, we can use 
%had to invoke analysis 
%in order to prove that there are only finitely many primes 
%(the same set for all test functions) that need to be excluded from the motivic expression 
%for the character in Theorem \ref{thm:distribution characters}(1), and so far that argument 
%only worked in characteristic zero.
%In order to prove Theorem \ref{thm:character}, we used only
%Theorem \ref{thm:distribution characters}(3), which did not require 
%this analytic argument. Now that we have proved Theorem \ref{thm:character}, 
%we know from the specialization principle \cite{CLF}*{Theorem 9.1.6} that 
%there exists $M>0$ such that 
%the specialization of the motivic constructible exponential function 
%to $\theta_{\pi}$ 
%works for all fields $\K\in {\mathcal A}_{\Z, M}$.
%	\Lxxx{Maybe I'm misunderstanding.  This used to be
%$\mc F_{\Z, M}$, so obviously there was a reason for
%changing it; but I thought that the point of this remark was
%precisely that things now worked independent of residual
%characteristic?}
%In Section \ref{section: application}, we will show that all the above 
%results are also true for $\K\in {\mathcal B}_{\Z, M}$, for sufficiently 
%large $M$.   
%\end{remark}

\subsection{Proof of Theorem \ref{thm:character} in the positive 
characteristic case}\label{sub: proof thm 3 positive char}
As above, $\pi$ is a restricted representation of minimal depth $r$, 
see Section \ref{sub: JKdatum}.

Let us summarize the relevant results from the previous sections.
Let $f_a$ be  any  constructible family of constructible 
motivic functions, with the parameter $a$ coming 
from some definable subassignment, 
and such that the support of $f_{a, \K}$ is contained in some fixed 
definable compact subset of $G(\K)_r\reg$ for all $\K\in {\mathcal F}_{\Z, M_1}$ 
for some $M_1$. 
Then: 
\begin{enumerate}
\item\label{item:family}  
There exist\out s a definable subassignment $B_{\fg, -r}\in \rde_{\Q}$, 
and $M>0$, such that for every  $\K\in {\mathcal F}_{\Z, M}$%
\fto{, for}{ and} every
restricted representation $\pi$ of $G(\K)$ of minimal depth $r$, 
the values $\Theta_{\pi}(f_{a, \K})$ 
depend only on \fto{a point}{$a$ and the image} $x\in B_{\fg, r}(\rf_{\K})$%
\ins{ of $\pi$ under the map of Theorem
\ref{thm:distribution characters}\pref{thm:distribution characters:B_g,-r}}%
, so we can 
\fto{refer to them as}{write instead} $\Theta_{x}(f_{a, \K})$.
\item
\out{The value} $\Theta_{x}(f_a)$ \out{, as a function of $a$,}
is a constructible 
motivic exponential function\ins{ of $a$}. We denote it by $F_x(a)$.
\end{enumerate}

Recall the family $\xi_{a,n}$ of definable test functions from 
Section \ref{subsub:thefamily},  that spans a dense subspace of $C_c^\infty(G(\K))$ for all 
$\K\in {\mathcal F}_{\Z, M_{\xi}}$ for the constant $M_{\xi}$ that depends only on $G$.
Let $M$ be the constant from the already proven characteristic zero case of 
Theorem \ref{thm:character}, and let $M_0=\max(M, M_\xi)$. 
Then, for the fields of characteristic zero $\K\in {\mathcal A}_{\Z, M_0}$ 
%	\Lxxx{What's $M$?  Do we mean $M_\eta$, or the
%constant $M$ referenced after the list below?}
in addition to the above two statements, we also have:
\begin{enumerate}\setcounter{enumi}{2}
\item\label{item:F}
 There exists a constructible motivic exponential function 
$F$ on $B_{\fg, r}\times G\reg$ such that for all $\K\in {\mathcal A}_{\Z, M_0}$,
$$
F_\K\ins{(x, \gamma)}=\frac{Q_{\fg, -r, \K}(x)}{\deg(x)}\theta_x(\gamma).$$

\item There exists a family of definable sets $U_{\gamma}$ 
(see Lemma \ref{lem:U_gamma}), with $\gamma\in G\reg$, such that 
for all fields $\K\in {\mathcal A}_{\Z, M_0}$, 
the character function $\theta_x\out{(\gamma)}$ is constant on $U_{\gamma, \K}$.  
	%% I changed $\theta_x(\gamma)$ to $\theta_x$, because
	%% it doesn't seem to me to make sense to say that
	%% $f(\gamma)$ (with $\gamma$ free) is constant on
	%% $S_\gamma$ (with $\gamma$ bound).
\end{enumerate}

%Let $M_0$ be the maximum of $M_{\eta}$ and the constant $M$
%that appears in
%\ins{the (already-proven) characteristic-zero case of}
%Theorem \ref{thm:character}%
%\out{ in the characteristic zero case that we have already proved}.

The Lemma below asserts that we can  transfer sufficient information about local constancy of the
character to the positive characteristic case, using everything we already 
proved in the  characteristic zero case.

\begin{lemma}\label{lem:local constancy}
Let $\K\in {\mathcal B}_{\Z, M_0}$, and let $\pi$ be a restricted 
representation of $G(\K)$ of minimal depth $r$. 
Then  the function $\theta_{\pi}\out{(\gamma)}$ is constant on the sets 
$U_{\gamma, \K}$. 
	%% As above.
\end{lemma}

\begin{proof}
	%% I shuffled some chunks of the proof, because it
	%% seems that no contradiction is necessary.
Let $x\in B_{\fg, -r}(\rf_\K)$ be the point that corresponds to the 
definable $r$-equivalence class of $\pi$.
Let $\gamma_0\in G(\K)_r\reg$.
\fto{
Since the function $\theta_{\pi}(\gamma)$ is locally constant, there 
exists the maximal open set $W_0$ containing $\gamma_0$ such that 
for $\gamma\in W_0$, 
we have $\theta_{\pi}(\gamma)=\theta_{\pi}(\gamma_0)$. 
We want to show that $W_0\supset U_{\gamma_0, \K}$. 
Suppose this is not so, and there exists a point
}{%
Suppose that
}
$g_0\in U_{\gamma_0, \K}$\fto{ such that}{, and put}
$\theta_{\pi}(g_0)=b\out{\neq \theta_{\pi}(\gamma_0)}$. 
Then by local constancy, there is a neighbourhood $U_b$ of $g_0$ such that
the character $\theta_{\pi}$ takes the value $b$ everywhere on $U_b$.
We can assume that $U_b\subset U_{\gamma_0, \K}$.   
Let ${\bf 1}_{U_b}$ be the characteristic function of the set $U_b$.
Let $\{\eta_a\}_{a\in A}$ be any family of definable test functions 
satisfying the 
conditions (1) and (2) 
of Definition \ref{def:definable equivalence}, for example, 
the family $\xi_{a,n}$ of Section \ref{subsub:thefamily}. 
Then there is a definable 
compact set $W_\K\subset G(\K)\reg$, containing $U_b$, and 
for every $\varepsilon >0$, a definable 
function $f_{\varepsilon}=\sum_{i=1}^mc_i\eta_{a_i, \K}$
(where the number $m$, the constants $c_i\in \C$, 
and the parameters $a_i\in A_\K$ depend on $\varepsilon$), 
such that $\supp (f_{\varepsilon})\subset W_\K$, and 
$$\sup_{W_\K}|{\bf 1}_{\U_b}-f_{\varepsilon}|<\varepsilon.$$

Let us evaluate the distribution character on the function 
$f_{\varepsilon}$.
On the one hand, by definition of $\theta_{\pi}$,
\begin{equation}\label{eq:one side}
\begin{aligned}
\left|
\Theta_{\pi}(f_{\varepsilon})- b\vol(U_b)\right|&=
\left|
\int_{G(\K)}\theta_{\pi}(g)f_{\varepsilon}(g)\, dg-
\int_{G(\K)}\theta_{\pi}(g){\bf 1}_{U_b}(g)\, dg
\right|\\
&<\varepsilon \vol(W_{\K})\sup_{W_\K}|\theta_{\pi}|.
\end{aligned}
\end{equation}
On the other hand, 
$\Theta_{\pi}(\eta_{a})$
is a constructible motivic function of $a$, 
and it can be specialized both to the fields in ${\mathcal A}_{\Z, M_0}$ 
and in 
${\mathcal B}_{\Z, M_0}$.
	{By the characteristic-zero case of Theorem
\ref{thm:character} (see \S\ref{sub: proof thm 3}), on}
 ${\mathcal A}_{\Z, M_0}$, its 
specialization coincides with the specialization of the function 
$$\int_{G\reg}\frac{\deg(x)}{Q_{\fg, -r}(x)}F(x,g)\eta_{a}(g)\, dg$$
of $a$, where $F$ is the function from condition (\ref{item:F}) above.

Therefore, by the transfer principle (see e.g. \cite{CHL}*{Theorem 2.7.2}),
the specializations of these two functions
\ins{to the fields in $\mathcal B_{\Z, M_0}$}
also coincide%
\out{ for the fields in ${\mathcal B}_{\Z, M_0}$}%
.
Hence, for every $i=0, \dots, m$, 
\begin{equation}
\Theta_{\pi}(\eta_{a_i, \K})=
\left(\frac{\deg(x)}{Q_{\fg, -r}(x)}\int_{G}
F(x,g)\eta_a(g)\, dg\right)_{\K}(a_i)=
\frac{\deg(x)}{Q_{\fg, -r, \K}(x)}\int_{G(\K)\reg}F_\K(x,g)\eta_{a_i}(g)\,dg.
\end{equation}

Since, for all $\K \in \mc A_{\Z, M}$,
the specialization of $F(x, g)$ to $\K$ is constant
for $g \in U_{\gamma, \K}$,
it also has to be constant 
for $\K\in {\mathcal B}_{\Z, M_0}$, and we have:

$$
%\frac{\deg(x)}{Q_{\fg, -r, \K}(x)}
\int_{G(\K)\reg}F_\K(x,g){\bf 1}_{U_b}(g)\,dg=
%\frac{\deg(x)}{Q_{\fg, -r, \K}(x)}
F_\K(x,\gamma_0)\vol(U_b).
$$

%$$
%\begin{aligned}
%\left(\int_{G\reg}\frac{\deg(x)}{Q_{\fg, -r}(x)}F(x,g)\eta_{a}(g)\, dg\right)_{\K}(a_i)\\= 
%\end{aligned}
%$$
%	\Lxxx{Is this really the formatting that we want?
%(Sorry, it's not a math question \ldots.)}

We get:
\begin{equation}\label{ineq: F and theta}
\left|\Theta_{\pi}(f_{\varepsilon})-
\frac{\deg(x)}{Q_{\fg, -r, \K}(x)}F_{\K}(x,\gamma_0)\vol(U_b)\right|<
\varepsilon \vol(W_\K)\sup_{g\in W_\K}F_\K(x, g)\frac{\deg(x)}{Q_{\fg, -r, \K}(x)}.
\end{equation}

We observe that when the field $\K$ is fixed, and the point 
$x\in B_{\fg, -r, \K}$ is fixed as well, while $\varepsilon$ 
and,  accordingly, $a_i$ and the function $f_{\varepsilon}$ vary,  
all the factors in the right hand side of the inequalities 
(\ref{ineq: F and theta}) and (\ref{eq:one side}), except $\varepsilon$, 
remain constant. 
Letting $\varepsilon \to 0$ and comparing (\ref{ineq: F and theta}) with (\ref{eq:one side}), we get
$$\theta_\pi(g_0) = b=\frac{\deg(x)}{Q_{\fg, -r, \K}(x)}F_\K(x, \gamma_0).$$
A similar argument shows that 
$\theta_{\pi}(\gamma_0)=\frac{\deg(x)}{Q_{\fg, -r, \K}(x)}F_\K(x, \gamma_0)=b$,
\fto	{which contradicts our assumption on $b$}
	{as desired}%
.
\end{proof}

\begin{remark}
A similar argument could be used to carry J.~Korman's local constancy result 
of \cite{korman:local-constancy} over to the fields of positive 
characteristic, when the residue characteristic is large enough.
On the other hand, this result in positive characteristic also follows from 
\cite{adler-korman:loc-char-exp}.  
In general, the fact that the characters are motivic\out, should have 
implications for their local constancy; %
\fto{, which remain}{ this remains} to be explored. 
\end{remark}

Now we can prove Theorem \ref{thm:character} in full generality, not just for 
characteristic zero fields: we just need to repeat the conclusion of the proof of Theorem 
\ref{thm:character} from Section \ref{subsub:end of proof}, which now works 
both in ${\mathcal A}_{\Z, M_0}$ and ${\mathcal B}_{\Z, M_0}$ by 
Lemma \ref{lem:local constancy}.

\section{Concluding remarks and future questions}

\subsection{Toward local integrability in positive characteristic}
\label{sub: local integrability}
We believe that just using the information about the distribution 
character on our family 
of functions, one can prove the local integrability of the character 
function%
\ins{ (at least in the range of validity of Theorem
\ref{thm:distribution characters}\pref{thm:distribution characters:motivic}%
)}%
.
So far 
we have shown that%
\ins{, for all fields $\K \in \mc F_{\Z, M_0}$,}
every definable $r$-equivalence class of restricted 
representations of $G(\K)$ of minimal depth $r$  corresponds to a point
$x\in B_{\fg, -r}(\rf_\K)$%
\out{ for all fields $\K\in {\mathcal F}_{\Z, M_0}$}%
, and the 
character $\theta_{\pi}$ on $G(\K)_r\reg$ is the
\ins{restriction of the}
specialization of the 
constructible motivic exponential function $F(x, g)$, up to a rational 
constant $\frac{\deg(x)}{Q_{\fg, -r, \K}(x)}$. 
Further, we know that for all fields $\K\in {\mathcal A}_{\Z, M_0}$
and $x \in B_{\fg, -r}(\rf_\K)$,
the 
specialization of $F$ is locally integrable with respect to $g$. 
We expect that a general ``transfer of integrability'' result should hold
for constructible motivic exponential functions\fto{, and}; this will be explored in 
 future work. 
Finally, once we \fto{have}{had} the local integrability result, we could 
conclude that the 
distribution character \fto{is}{was} a continuous linear functional on the space 
$C_c^{\infty}(G(\K)_r)$ for all $\K\in {\mathcal F}_{\Z, M}$, and therefore knowing its values on a definable family of definable test functions would be sufficient. 
%Then the same argument as in Section \ref{sec:badprimes} would 
%show\out s that \out{in fact} the notion\ins s of definable $r$-equivalence
%\fto	{coincides with the notion of $r$-equivalence}
%	{and $r$-equivalence coincided}%
%, \fto{and in fact}{so that} our definable subassignment $B_{\fg, -r}$
%\fto{parameterizes}{parameterized} the $r$-equivalence classes of representations, both in 
%characteristic zero, and in positive characteristic, as long as the 
%residue characteristic \fto{is}{was} large enough. 

\subsection{Toward an explicit algorithm for computing character values}
Suppose one wanted to write a computer program that
\fto{takes}{took} the residue field characteristic $p$ and the depth $r$ 
as an input, and 
\fto{produces}{produced}  a number $M$ and 
the character tables for representations of depth $r$ of 
$\operatorname{Sp}_{2N}(\Q_p)$ and $\operatorname{SO}_n(\Q_p)$, for $p>M$ 
(here we take $\Q_p$ for simplicity so that we do not have to think of the implementation of the definition of a field extension; all our arguments work for an arbitrary $p$-adic field.
%\ins{---see the proof of Lemma \ref{lem:family of functions}}%
%).   
Here is what we can contribute so far towards such a program.
%\Jxxx{I should write more about explicit steps towards an algorithm}
\begin{enumerate}
\item It is easy to write a program that would generate the 
parameter space $B_{\fg, -r}$ following the recipe given here and in 
\cite{cunningham-hales:good}. No integration is required.

\item If motivic integration had been implemented, then one could further 
take a formula defining the test function $f$ as input, and for every such 
function with support contained in $G(\Q_p)_r\reg$, 
output  the actual value of the distribution character evaluated at 
this function.  
%\Lxxx{We would also need to know the field \K---so far,
%since we don't have a motivic theory of distributions, it
%doesn't make sense to define the value of a distribution at
%an element of $\mc C(\mathrm{whatever})$.
%We also need to make sure that the support of $f$ (or its
%specialization, whichever we meant) lay in $G(\K)_r\reg$.
%Julia says: we've just restricted to $\Q_p$ for this reason. I have taken care of the support.}
\end{enumerate}
This already would make a lot of experimentation with characters possible. In particular, one could
\ins{(near the identity)}
test the character for local constancy on specific sets, as long as those sets were definable;
\out{one could} test various linear combinations of characters for stability; etc.

\begin{enumerate}
\addtocounter{enumi}2
\item Since we know that the character is a constructible function, we conclude that if an algorithm for cell decomposition had been 
implemented, we could produce, for each \out{given} depth $\fto s r$,
	%% $r$ instead of $s$ for consistency with notation
	%% elsewhere.
a complete list of 
neighbourhoods of regular elements of \out a depth \ins{at least}
$\fto s r$ 
\out{in our neighbourhood of the identity}
	%% I think that the ``of depth at least $r$''
	%% already specifies where we are in relation to the
	%% identity.
on which the function 
$\theta_{\out{\pi_}x}$ is constant, and the list of
\ins{its} values \out{of this function}
on \fto{these}{those} neighbourhoods.

\end{enumerate}

\subsection{Future questions}
There are a few things pleading
	%% :-)
to be carried out in the context of this paper, and we plan to address them in future work.
First, there is folklore evidence that it should be possible to improve the 
local constancy result of \cite{korman:local-constancy}, and our 
Theorem \ref{thm:character}
corroborates that:
\fto	{%
for a constructible exponential function, there has to be 
a definable family of subsets that form a partition of the domain, such that
the function  has to be constant on these sets%
}{%
the level sets of a constructible exponential function
form a partition of its domain into definable subsets%
}%
, and it seems that the 
sets $\gamma T\ins{(\K)}_{r'+s(\gamma)}$ are not quite large enough to
\fto{form a suitable partition}{be definable}
because of the ``extra'' term $s(\gamma)$. 
We hope to obtain a rigorous proof along these lines.

Second, it is natural to investigate the character away from the identity. 
It is clear that the multiplicative characters of the field need to be 
included in the language in order to make similar treatment possible%
\fto{, but it is very likely}{.  We expect}
	%% The `but' makes it sound like the two statements
	%% are in opposition, but it seems to me that they're
	%% not.
that, once this is done, it will be possible to describe the
\ins{values of the} Harish-Chandra character \out{values}
on the whole group.

\end{document}